\documentclass[11pt]{article}
\usepackage{amsmath, amssymb, amsthm}
\DeclareMathOperator{\E}{\mathbb{E}}
\DeclareMathOperator{\spn}{span}
\DeclareMathOperator{\rank}{rank}
\newcommand{\supp}{\mathrm{supp}\,}
\def\T{\sf T}
\renewcommand{\P}{\mathbb{P}}
\newcommand{\C}{\mathbb{C}}

\newcommand{\R}{\mathbb{R}}

\def\eps{\epsilon}
\def\phi{\varphi}
\newcommand{\rmd}{\,\mathrm{d}}

\def\M{{\rm M}}
\def\e{\vec{{\bf e}}}
\def\SSS{S}
\def\Area{\mathcal{\bf A}}

\usepackage{graphicx}
\usepackage{hyperref}
\newtheorem{theorem}{Theorem}[section]
\newtheorem{proposition}[theorem]{Proposition}
\newtheorem{lemma}[theorem]{Lemma}
\newtheorem{definition}[theorem]{Definition}
\newtheorem{conjecture}{Conjecture}

\newtheorem{assumption}{Assumption}
\newtheorem{remark}[theorem]{Remark}

\begin{document}

\title{A universal result for consecutive random subdivision of polygons}

\author{
Nguyen Tuan Minh\footnote{Center for Mathematical Sciences, Lund University, S\"olvegatan 18, Box 118, 22100 Lund, Sweden.} ${}^{,}$\footnote{Email address: nguyen@maths.lth.se}
\, and 
Stanislav Volkov$^{*,}$\footnote{Email address: s.volkov@maths.lth.se
}}
%
%

\maketitle
\begin{abstract} We consider consecutive random subdivision of polygons described as follows. Given an initial convex polygon with~$d\ge 3$ edges, we choose a point at random on each edge, such that the proportions in which these points divide edges are i.i.d.\ copies of some random variable~$\xi$. These new points form a new (smaller) polygon. By repeatedly implementing this procedure we obtain a sequence of random polygons. The aim of this paper is to show that under very mild non-degenerateness conditions on~$\xi$, the shapes of these polygons eventually become ``flat''  The convergence rate to flatness  is also investigated; in particular, in the case of triangles~($d=3$), we show how to calculate the exact value of the rate of convergence, connected to Lyapunov exponents. Using the theory of products of random matrices our paper greatly generalizes the results of~\cite{VOL} which are achieved mostly by using ad hoc methods.
\end{abstract}

\begin{keywords}
Random subdivisions, products of random matrices, Lyapunov exponents.
\end{keywords}

\begin{amsc}
60D05, 60B20, 37M25.
\end{amsc}

\section{Introduction}

Many problems of consecutive random subdivision of a convex geometrical figure have been investigated by several authors since 1980s. In~\cite{WAT}, G.~S.~Watson introduced the following model: given an initial triangle, one chooses a point on each edge by keeping the same random proportion~$\xi$ and hence obtaining a new triangle. If one repeats the above process with independent identically distributed random proportions~$\xi^{(n)},n=1,2,\dots$ then the limit triangle vanishes to the centroid of the initial triangle. To study the {\it shapes} of these triangles, let us rescale the newly formed in each step triangle in such a way that the largest side has length~$1$. It is interesting that the ``limit'' of these rescaled triangles is non-vanishing and, in fact, random. Veitch and Watson in~\cite{VW} also gave an extension for a system of points in higher dimensional real space. With the same motivation of random triangles, Mannion in~\cite{MAN} studied the situation where on each step the triangle is formed by choosing three uniformly distributed random points inside the {\it interior} of the preceding triangle. The sides of these triangles almost surely converge to collinear segments. Diaconis and Miclo~\cite{DIAMIC} considered a triangle split by the three medians such that one of the~$6$ triangles is chosen at random to replace the original triangle. It turns out that the limiting triangle's shape is flat.
Volkov in~\cite{VOL} discovered a similar phenomenon by considering a model where the new triangle is formed by choosing a random point uniformly and independently on each of the sides of the original triangle; he also studied  distribution of the ``middle'' point.
 
In the present paper, we give a generalization  of Volkov's result in~\cite{VOL} for all convex polygons and nearly all non-degenerate distributions of proportions in which  the sides of the polygon are split. 

Let us now formulate the model rigorously. Fix~$d\ge 3$ and  a random variable~$\xi$ whose support lies on~$[0,1]$. Let~$L_0=A_1^{(0)}A_2^{(0)}\dots A_d^{(0)}$ be a convex~$d$-polygon (i.e., a convex polygon with~$d$ sides) in the plane, with edges~$A_j^{(0)} A_{j+1}^{(0)}$, $j=1,2,\dots,d$, with the convention~$A_{d+1}^{(1)}\equiv A_{1}^{(1)}$. Randomly choose a point~$A_j^{(1)}$ in~$A_j^{(0)}A_{j+1}^{(0)}$ such  that~$|A_j^{(0)} A_j^{(1)}| / |A_j^{(0)}A_{j+1}^{(0)}|=\xi_{i}$, where~$\xi_{i}$, $i=1,\dots,d$, are i.i.d.\ copies of the random variable~$\xi$. Thus we obtain new convex polygon~$L_1=A_1^{(1)}A_2^{(1)}\dots A_d^{(1)}$. Repeating the above process such that the random vectors 
$\left(\xi_1^{(n)},\xi_2^{(n)},\dots,\xi_d^{(n)}\right)$, $n=1,2,\dots$, are i.i.d., we obtain a Markov chain of polygons~$(L_n)_{n\ge 0}$ where~$L_n=A_1^{(n)} A_2^{(n)} \dots A_d^{(n)}$. 
\begin{figure}[!ht]    \centering
    \includegraphics[scale=1]{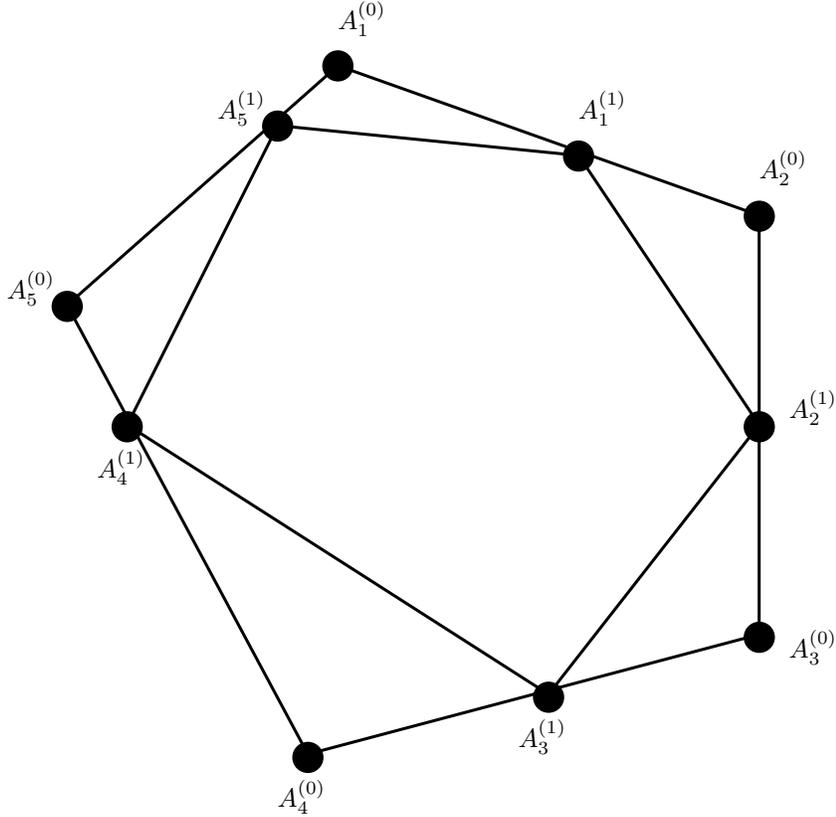} 
    \caption{A new smaller random pentagon~$L_1$ obtaining from the primary pentagon~$L_0$.}
\end{figure}

It is easy to see that the polygons~$L_n$ become smaller and smaller and eventually converge to a point, however the behaviour of their {\it shapes} is less clear. To study the shapes we may, for example,  place one of the vertices at the origin~$(0,0)$ and rescale the polygon in such a way that its longest edge has always length~$1$. We will show that under some regularity conditions on the distribution of~$\xi$ the rescaled polygon will eventually become degenerate, i.e.\ flat, in the sense that all of its vertices will be lying approximately along the same line; observe that this is equivalent to the fact that the area of the rescaled polygon converges to~$0$ as~$n$ goes to infinity.

Let~$l_j^{(n)}=A_j^{(n)} A_{j+1}^{(n)}$, $j=1,2,\dots,d$,  be the vector corresponding to the~$j$-th side of~$L_n$ and~$(x_j^{(n)},y_j^{(n)})$  denote its Cartesian coordinates. From elementary geometrical calculations one can obtain the following linear relation:
\begin{equation}\label{linear}
x^{(n+1)}=H_{n+1} x^{(n)}, \ y^{(n+1)}=H_{n+1} y^{(n)}
\end{equation}
where~$x^{(n)}=\left(x_1^{(n)},  x_2^{(n)}, \dots , x_d^{(n)} \right)^{\T}$ 
and~$y^{(n)}= 
\left(y_1^{(n)}, y_2^{(n)}, \dots , y_d^{(n)} \right)^{\T}$ are column vectors, and~$H_n$ is an i.i.d.\ copy of the following random matrix
\begin{equation}\label{matrix1}
H=H(\xi_1,\dots,\xi_d)=  \left(\begin{matrix}
1-\xi_1 & \xi_2 & 0 & \dots & 0\\
0 & 1-\xi_2 & \xi_3 & \dots &  0\\
\vdots &\vdots & \vdots &\ddots & \vdots \\
 0 & 0& 0&  \ddots & \xi_d \\
 \xi_1 & 0& 0& \dots  & 1-\xi_d 
\end{matrix}\right)
\end{equation}
and~$\xi_1,\xi_2,\dots,\xi_d$ are i.i.d.\ copies of a random variable~$\xi$. Note that~$\sum_{j=1}^d x_j^{(n)}=0$ and~$\sum_{j=1}^d y_j^{(n)}=0$.
In particular,
$l_j^{(n)}=(\e_j H^{(n)} x^{(0)},\e_j H^{(n)} y^{(0)})$ where~$H^{(n)}=H_n H_{n-1} \dots H_1$ and~$\e_j=(0,\dots,0,1,0,\dots,0)$ is~$1\times d$ vector with~$1$ on the~$j$-th place. Note also that if the original polygon is non-degenerate then~$H^{(n)} x^{(0)}$ and~$H^{(n)} y^{(0)}$ are non-zero vectors for any~$n$.

To ensure that~$L_n$ is a non-degenerate convex polygon and that the subdivision is genuinely random, we need the following 
\begin{assumption}\label{Asu1}
$\P(\xi\in \{0,1\})=0$ and the support of~$\xi$ contains at least two distinct points in~$(0,1)$, i.e.\ the distribution of~$\xi$ is non-degenerate.
\end{assumption}

We can define  ``thickness'' of a two-dimensional object as the smallest possible ratio between  its one-dimensional projections on the two coordinate axes of a Cartesian coordinate system (where we can orient this system arbitrarily); this quantity always lies between~$0$ and~$1$; moreover, it equals  one for a circle, and it equals  zero for any segment. The sequence of~$L_n$ converges to a ``flat figure", or simply to ``flatness", if the sequence of its thicknesses converges to zero. In the case of polygons, this definition is equivalent to
\begin{definition}
We say that the sequence of polygons~$L_n$ converges to a flat figure as~$n\to\infty$ if
$$
\lim_{n\to\infty}\frac{\Area(L_n)}{\left(\max_{j=1,\dots,d} \Vert l_j^{(n)}\Vert\right)^2}=0.
$$
Here~$\Area(L_n)$ denotes the area of the polygon~$L_n$.
\end{definition}

The main purpose of our paper is to (partially) establish the following  phenomenon.
\begin{conjecture}\label{conj}
Suppose that Assumption~\ref{Asu1} holds, then the sequence of polygons~$L_n$ converges to a flat figure almost surely as~$n\to\infty$.
\end{conjecture} 

Further the dynamics of the random subdivisions will be formulated as a certain model related to products of random matrices and its point limit in the projective space. Let~$\R^{d}$ (and~$\mathbb{C}^{d}$)  denote the linear space of all~$d$-dimensional real (complex, resp.) column vectors under the field of real (complex) numbers. The real (complex) projective space~$P(\R^d)$ is defined as the quotient space~$(\R^d\setminus\{0\})/\sim$, where~$\sim$ is the equivalence relation defined by~$x\sim y$, $x,y\in \R^d$ if there exists a real (complex) number~$\lambda$ such that~$x=\lambda y$. We denote~$\overline{x}$ as the equivalence class of~$x$. The projective space~$P(\R^d)$ becomes a compact metric space if we consider the following ``angular" metric
\begin{equation}\label{metric1}
\displaystyle \delta(\overline{x},\overline{y})=\sqrt{1-\frac{(x,y)^2}{||x||^2.||y||^2}}.
\end{equation}
where~$||\cdot ||$ and~$(\cdot,\cdot)$ are respectively the Euclidean norm and the Euclidean scalar product on~$\R^d$. 
One can see that~$\delta(\overline{x},\overline{y})$ is actually the sinus of the smaller angle between the lines corresponding to~$\bar x$ and~$\bar y$.

Next, each linear mapping~$A:\R^d\to \R^d$ can be generalized to~$P(\R^d)$ by setting
$$
A\overline{x}=\overline{Ax}
$$ 
for every~$x\in \R^{d}\setminus{\sf Ker}(A)$.
Let us also  define
\begin{align}\label{eqlspace}
{\cal L}=\{ v \in \R^d : v_1+v_2+\dots+v_n=0\}.
\end{align}
Observe that since~$\sum_{j=1}^d x_j^{(n)}=0$, $\sum_{j=1}^d y_j^{(n)}=0$, we have~$x^{(n)},y^{(n)}\in{\cal  L}$.

\begin{proposition}\label{lemflat}
Suppose that \begin{equation}\label{flatness}
\lim_{n\to\infty}\delta\left( H^{(n)}\overline{x},H^{(n)}\overline{y}\right)= 0 
\end{equation}
almost surely for every~$x, y\in L_n$ such that~$(x_1,y_1)$, $(x_2,y_2)$,$\dots$,$(x_d,y_d)$ are coordinates of vectors corresponding to consecutive edges of the convex~$d$-polygon in the real plane. Then~$L_n$ converges to a flat figure as~$n\to\infty$.
\end{proposition}
\begin{proof}
Using the formula for~$\delta\left(\overline{x^{(n)}},\overline{y^{(n)}}\right)$ and omitting the superscript~${}^{(n)}$ for all~$x^{(n)}$ and~$y^{(n)}$ for simplicity, we obtain that
$$
\delta(\overline{x},\overline{y})^2=\frac{\left(\sum_{i=1}^d x_i^2\right)\left( \sum_{i=1}^d y_i^2\right)-\left(\sum_{i=1}^d x_i y_i\right)^2}
{\left(\sum_{i=1}^d x_i^2\right)\left( \sum_{i=1}^d y_i^2\right)}
=\frac{\sum_{1\le i<j\le d} (x_i y_j-x_j y_i)^2}
{\left(\sum_{i=1}^d x_i^2\right)\left( \sum_{i=1}^d y_i^2\right)}
=:\delta_n
$$
where~$\delta_n\to 0$ a.s.

According to a well-know formula for the signed area~$\Area(L)$ of a planar non-self-intersecting polygon~$L$ with vertices~$(a_1,b_1)$, $\dots$, $(a_d,b_d)$, see~\cite{Beyer}
$$
2 \Area(L)=\det\begin{pmatrix}a_1 &a_2\\b_1&b_2\end{pmatrix}
+\det \begin{pmatrix}a_2 &a_3\\b_2&b_3\end{pmatrix}
+\dots+
\det \begin{pmatrix}a_d &a_1\\b_d&b_1\end{pmatrix}.
$$
Since we know only the coordinates of the vectors forming the edges of polygon~$(x_i,y_i)$, $i=1,2,\dots,d$ with the obvious restriction~$\sum_{i=1}^d x_i=\sum_{i=1}^d y_i=0$, we can assume that the polygon's vertices have the coordinates
\begin{align*}
a_i=x_1+\dots+x_i,\\
b_i=y_1+\dots+y_i,
\end{align*}
$i=1,2,\dots,d$, thus yielding that~$a_d=b_d=0$ so that the last two determinants in the formula for~$2\Area(L)$ are~$0$, and hence
\begin{align*}
2\Area(L)&=\sum_{i=1}^{d-2} \det\begin{pmatrix}
a_i & a_{i+1}\\ b_i & b_{i+1}
\end{pmatrix}
=\sum_{i=1}^{d-2} 
\det\begin{pmatrix}a_i & a_i+x_{i+1}\\ b_i & b_i+y_{i+1}\end{pmatrix}
=\sum_{i=1}^{d-2} 
(a_i y_{i+1}-b_i x_{i+1})
\\
&=\left[x_1 y_2+(x_1+x_2)y_3+\dots+(x_1+x_2+\dots x_{d-2})y_{d-1}\right]\\
&-\left[y_1 x_2+(y_1+y_2)x_3+\dots+(y_1+y_2+\dots y_{d-2})x_{d-1}\right]
\\
&=\sum_{1\le i<j\le d-1} \det\begin{pmatrix}
x_i &y_i\\ x_j &y_j
\end{pmatrix}.
\end{align*}
Therefore the area~$\Area(L_n)$ of the polygon~$L_n$ satisfies
\begin{align*}
& |2\Area(L_n)|=\left| \sum_{1\le i<j\le d-1} \det 
\begin{pmatrix}
x_i &y_i\\ x_j &y_j \end{pmatrix} \right|
\le \sum_{1\le i<j\le d-1} 
\left|\det \begin{pmatrix}x_i &y_i\\ x_j &y_j\end{pmatrix} \right|\\ &
\le\sqrt{\sum_{1\le i<j\le d} (x_i y_j-x_j y_i)^2}
=\sqrt{\delta_n \left(\sum_{i=1}^d x_i^2\right)\left( \sum_{i=1}^d y_i^2\right)}. 
\end{align*}
Consequently,
\begin{align*}
\frac{\Area(L_n)}{\left(\max_j \Vert l_j^{(n)}\Vert\right)^2}
\le
\frac{1}{2}\sqrt{\delta_n \frac{\left(\sum_{i=1}^d x_i^2\right)\left( \sum_{i=1}^d y_i^2\right)}{\left(\max_{j=1,\dots,d} \left[x_j^2+y_j^2\right]\right)^2}}
\le \frac 12 \sqrt{\delta_n\cdot d \cdot d}\to 0
\end{align*}
since~$x_i^2\le \max_{j=1,\dots,d} (x_j^2+y_j^2)$ for each~$i$, and the same holds for~$y_i$.
\end{proof}

Note that~${\cal L}$ defined by~\eqref{eqlspace} is an invariant subspace of~$H$.
Therefore, we can restrict the linear  transformation~$H$ to~$\R^{d-1}$ by considering only the first~$d-1$ coordinates of~$x$ and~$y$ respectively. One can easily deduce that the restriction of the transformation~$H$ can be described by the~$(d-1)\times (d-1)$ matrix 
\begin{equation}\label{matrix2}
T=T(\xi_1,\dots,\xi_d)=  \left(\begin{matrix}
1-\xi_1 & \xi_2 & 0 & \dots & 0& 0\\
0 & 1-\xi_2 & \xi_3 & \dots & 0& 0\\
\vdots &\vdots & \vdots &\ddots & \vdots &\vdots \\
 0           & 0& 0& \dots  & 1-\xi_{d-2} & \xi_{d-1}\\
 -\xi_d & -\xi_d& -\xi_d& \dots & -\xi_d & 1-\xi_{d-1}-\xi_d 
\end{matrix}\right)
\end{equation}
and then the linear relation~\eqref{linear} still has the same formulation in~$\R ^{d-1}$ for~$T$. The condition~\eqref{flatness} for the matrix~\eqref{matrix2} now can be restated as
\begin{proposition}\label{propcond}
Let~$\{T_n\}_{n\ge1}$ be a sequence of random matrices, which are independent copies of the matrix~$T$ in~\eqref{matrix2} and let~$T^{(n)}=T_nT_{n-1}....T_2T_1$. Assume that
\begin{align}\label{cond}
\lim_{n\to\infty}\ 
\delta(T^{(n)}\bar x,T^{(n)}\bar y)= 0  
\end{align}
almost surely for any~$x=(x_1,...,x_{d-1})^{\T}  ,y=(y_1,...,y_{d-1})^{\T}  \in \R^{d-1}$, such that~$(x_1,y_1), (x_2,y_2),...,$ $(x_{d-1},y_{d-1})$ are coordinates of~$d-1$ consecutive edges of a convex~$d$-polygon in the real plane. Then~$L_n$ converges to a flat figure as~$n\to\infty$.
\end{proposition}
\begin{proof}
Basically, we need to show the following geometric fact. Suppose that~$x^{(n)}=(x^{(n)}_1,\dots,x^{(n)}_{d-1})$
and~$y^{(n)}=(y^{(n)}_1,\dots,y^{(n)}_{d-1})$.
are such that~$\delta_n:=\delta(\overline{x^{(n)}},\overline{y^{(n)}})\to 0$ as~$n\to\infty$, then~$\tilde\delta_n:=\delta(\overline{\tilde x^{(n)}},\overline{\tilde y^{(n)}})\to 0$, 
where~$\tilde x^{(n)}=(x^{(n)}_1,\dots,x^{(n)}_d)$ and~$\tilde y^{(n)}=(y^{(n)}_1,\dots,y^{(n)}_d)$ 
with~$x^{(n)}_d=-\sum_{i=1}^{d-1} x^{(n)}_i$,
$y^{(n)}_d=-\sum_{i=1}^{d-1} y^{(n)}_i$,
for all~$n$. Observe that~$\delta_n$ and~$\tilde\delta_n$ represent the angular distance on the spaces~$P(\R^{d-1})$ and~$P(\R^{d})$ respectively.

Indeed, suppose that~$\delta_n<\epsilon$ for some very small~$\epsilon>0$. Let us from now on also omit the superscript~${}^{(n)}$ as this does not create a confusion. Without loss of generality we can assume that~$\| x\|=\|y\|=1$, that is, $\sum_{i=1}^{d-1} x_i^2=1=\sum_{i=1}^{d-1} x_i^2$. Denote by~$c=(x,y)=\sum_{i=1}^{d-1} x_i y_i=\cos(x,y)$, so that~$c^2+\delta_n^2=1$.
We have
\begin{align}\label{eqdeltadelta}
\tilde\delta_n^2 &=
\frac{(1+x_d^2)(1+y_d^2)-(\sum_{i=1}^d x_i y_i)^2}
{(1+x_d^2)(1+y_d^2)}
=
\frac{(1+x_d^2)(1-c^2)+( y_d -c x_d)^2}
{(1+x_d^2)(1+y_d^2)} \nonumber
\\&
\le (1-c^2)+( y_d -c x_d)^2
=\delta_n^2+\left(\sum_{i=1}^{d-1} u_i\right)^2
\end{align}
where~$u=y-cx=(y_1-c x_1,\dots,y_{d-1}-c x_{d-1})$ is the difference between vector~$y$ and the projection of~$y$ on~$x$. Consequently, $u$ is orthogonal to~$x$ and~$\|u\|^2=\|y\|^2-\|cx\|^2=1-c^2=\delta_n^2$.
By the inequality between the quadratic and arithmetic means~$|\sum_{i=1}^{d-1} u_i|^2\le (d-1) \| u\|^2$ hence~\eqref{eqdeltadelta} implies that~$\tilde\delta_n^2\le [1+(d-1)] \delta_n^2 \le d\epsilon^2$.
\end{proof}

The rest of the paper is organized as follows.
In Section~\ref{Secd=3}, by applying the classical Furstenberg's theorem for products of~$2\times 2$ invertible random matrices, we will show that~\eqref{cond} is fulfilled for~$d=3$ (Theorem~\ref{mainthm1}). In a higher dimensional case, it is necessary to show that the closed semigroup generated by the support of the random matrix~$T$ is strongly irreducible and contracting. We will  show that~\eqref{cond} holds for any {\it odd} number~$d>3$ in Section~\ref{Secd>=4}. For the remaining case when~$d\ge 4$ is {\it even}, we will have to require that the random matrix~$T$ in~\eqref{matrix2} is invertible almost surely. We actually believe that this extra requirement is not really needed, however we are unable to show the result without this extra condition. The results are summarized in Theorem~\ref{mainthm2}. The exponential rate of convergence of random polygons will be considered in Section~\ref{Secrate}, see   Theorems~\ref{mainthm3}, \ref{mainthm4} and~\ref{mainthm5}. 

Finally, in Section~\ref{sec_gener} we mention some generalizations of our model, as well as open problems.
Also note that throughout the the paper we denote by~$GL(d,\R)$ the group of~$d\times d$ invertible matrices of real numbers and~$SL^{\pm}(d,\R)$ the closed subgroup of~$GL(d,\R)$ containing all matrices with determinant~$+1$ or~$-1$.

\section{Random subdivision of triangles (\texorpdfstring{$d=3$}{})}
\label{Secd=3}
\begin{proposition} {\rm(Furstenberg's theorem, Theorem~II.4.1 in~\cite{BL}, page~30)}  \label{thmFur}
Let~$\mu$ be a probability measure on~$GL(2,\R)$ and~$G_{\mu}$ be the smallest closed subgroup of~$GL(2,\R)$ which contains the support of~$\mu$. Suppose that the following hold: 
\begin{itemize}
\item[(i)]  $G_{\mu}\subset SL^{\pm}(2,\R)$;
\item[(ii)] $G_{\mu}$ is not compact;
\item[(iii)] There does not exist any common invariant finite union of one-dimensional subspaces of~$\R^{2}$ for all matrices of~$G_{\mu}$.
\end{itemize}
Let~$\{M_n, n\ge 1\}$ be a sequence of independent random matrices with distribution~$\mu$ and~$\overline{ x},\overline{y}\in P(\R^2)$. Then
$$
\lim_{n\to\infty}\delta\left(M_nM_{n-1}...M_1\overline{ x},M_nM_{n-1}...M_1\overline{y}\right)=0.
$$ 
\end{proposition} 

Note that when~$M_1$ is invertible almost surely and~$\det(M_1)$ is possibly not equal to~$\pm 1$, it is enough to verify the above conditions for the group~$G_{\tilde\mu}$ generated by all~$\tilde{M}=(\det M)^{-1/2} M$, where~$M$ is any invertible matrix in the support of~$\mu$. 

\begin{theorem}\label{mainthm1}
Conjecture~\ref{conj} is fulfilled for~$d=3$.
\end{theorem}
\begin{proof}
When~$d=3$ the random matrix~$T$ equals
$$
T=T(\xi_1,\xi_2,\xi_3)=\begin{pmatrix}
1-\xi_1 & \xi_2\\
-\xi_3 &1-\xi_2-\xi_3
\end{pmatrix}
$$
where~$\xi_1,\xi_2,\xi_3$ are i.i.d.\ copies of~$\xi$. Let~$\mu$ be the probability measure associated with the random matrix~$T(\xi_1,\xi_2,\xi_3)$ . Observe that~$\det(T)=\xi_1\xi_2\xi_3+(1-\xi_1)(1-\xi_2)(1-\xi_3)>0$ as long as~$\xi_1,\xi_2,\xi_3\in (0,1)$,  thus~$\tilde T=(\det T)^{-1/2} T$ is a.s.\ well-defined. Let~${G}_{\mu}$ be the group generated by all the invertible matrices in the support of~$\mu$ and~$G_{\tilde\mu}$ be the group generated by all~$\tilde{T}$, where~$T\in{G}_{\mu}$.  
Since~$\det(\tilde T(\xi_1,\xi_2,\xi_3))=1$ for all possible~$\xi_1,\xi_2,\xi_3$ and the determinant of a product of two matrices equals the product of their determinants, we have~$\det(\tilde{T})=1$ for all~$\tilde{T}\in G_{\tilde{\mu}}$. Consequently, condition~(i) of Proposition~\ref{thmFur} is fulfilled.

Now let us verify condition~(ii), i.e.\ that the group~$G_{\tilde{\mu}}$ is not compact. From Assumption~\ref{Asu1} it follows that we can choose~$a,b\in\supp\xi$ such that~$a,b\in(0,1)$ and~$a\ne b$. Let 
\begin{align}\label{Qt}
Q=T(a,b,a)\ T(a,b,b)^{-1}
\ T(b,a,b) \ T(b,a,a)^{-1}
=\begin{pmatrix}
1 & 0   \\  t & 1
\end{pmatrix}, 
\end{align}
where~$$t=-\frac{(a-b)^2}{2ab+b^2-a-2b+1}.$$
Since~$a\ne b$ and~$2ab+b^2-a-2b+1=(a+b-1)^2+a(1-a)>0$ the quantity~$t$ is well-defined and negative.
Observe that~$Q\in G_{\tilde{\mu}}$ and hence 
$$
Q^m=\begin{pmatrix}
1 & 0   \\  mt & 1
\end{pmatrix} \in G_{\tilde{\mu}}
$$ 
as well. Since~$||  Q^m ||\sim m\to \infty$ as~$m\to\infty$,  the group~$G_{\tilde{\mu}}$ is indeed not compact.

Finally, we need to check the condition~(iii) of Theorem~\ref{thmFur}, that is,  that~$G_{\tilde{\mu}}$ is strongly irreducible, or equivalently that~$G_{\mu}$ is strongly irreducible. Suppose the contrary, i.e.\ there is a union~$L$ of one-dimensional subspaces of~$\R^2$ such that~${T}(L)=L$ for any~$\tilde{T}\in G_{\tilde{\mu}}$. 
Let~$L=V_1\cup V_2 \cup \dots \cup V_k$, $k\ge 1$.

First, suppose that~$L$ contains a vector of the form~$(x,y)^{\T}$ such that~$x\ne 0$. Then at least one of~$V_i$ is the linear span of~$v=(1,r)^{\T}$, $r\in\R$; without loss of generality let this be~$V_1$. Since~$Q$ defined by~\eqref{Qt} belongs to~$G_{\mu}$, for all~$m=1,2,\dots$ we must have~$Q^m\in G_{\mu}$  and  thus~$Q^m L \subseteq L$. The latter implies that~$v_m:=Q^m v\in L$.
However, the slopes of the vectors~$v_m$
equal~$mt+r$ which  take distinct values for different values of~$m$,  therefore~$L$ cannot be a union of a {\em finite} number of linear subspaces, leading to a contradiction. 

Therefore, the only candidates for~$V_i$ can be linear spaces spanned by~$(0,1)^{\T}$. To show that this is not possible either, pick any~$a\in (0,1)^{\T}$ which is in the support of~$\xi$, then
$$
T(a,a,a)\begin{pmatrix}
0\\ 1
\end{pmatrix} =
\begin{pmatrix}a  \\ 1-2 a\end{pmatrix}\in L.
$$
Hence there must be a vector in~$L$ whose first coordinate is non-zero, which leads to the situation already considered above.

Consequently, the conditions of the Furstenberg's theorem~\ref{thmFur} are fulfilled, implying a.s.\  convergence to flatness in case~$d=3$.
\end{proof}

\section{General case (\texorpdfstring{$d\ge4$}{}) } 
\label{Secd>=4}
We start with a few definitions.
\begin{definition}
We say that a family~$\mathcal{H}$ of~$d\times d$ matrices is irreducible in~$\R^d$ if there exists no proper linear subspace~$L$ of~$\R^d$ such that~$H(L)=L$ for all~$H\in \mathcal{H}$.
\end{definition}
\begin{definition}
We say that a family~$\mathcal{H}$ of~$d\times d$ matrices is {\em strongly} irreducible in~$\R^d$ if there exists no union~$L$ of finite number of proper linear subspaces of~$\R^d$ such that~$H(L)=L$ for all~$H\in \mathcal{H}$.
\end{definition}

\begin{definition}
We say a family~$\mathcal{H}$ of~$d\times d$ matrices  has {\em contraction property} if there is a sequence of elements~$\{A_n\}_{n\ge 1}\subset \mathcal{H}$ such that~$||A_n||^{-1}A_n$ converges to a rank one matrix. 
\end{definition}

We will make use of the following
\begin{proposition}[Theorem~III.4.3 in~\cite{BL}, p.~56]\label{th343}
Let~$A_i$ be a sequence of i.i.d.\ random matrices in~$GL(d,\R)$ with common distribution~$\mu$. Let~$\SSS_\mu$ be the smallest closed semigroup generated by its support.
Suppose that~$\SSS_\mu \subset GL(d,\R)$ is strongly irreducible and contracting. Then for any~$\overline{x},\overline{y}\in{P}(\R^d)$
$$
\lim_{n\to\infty} \delta(A_n\dots A_1 \overline{x}, A_n\dots A_1 \overline{y})=0\text{ a.s.}
$$
\end{proposition}
Note that, when~$A_1$ is only invertible almost surely, it is enough to verify the strong irreducibility and contraction condition for the semigroup~$\tilde{S}_{\mu}$ generated by all~$\tilde{A}=(|\det A|)^{-1/d} A$, where~$A$ is any invertible matrix in the support of~$\mu$.
In our case the measure~$\mu$ corresponds to random matrices of type~$T=T(\xi_1,...,\xi_d)$ defined by~\eqref{matrix2}. Observe that
$$
\det(T)=\prod_{i=1}^d (1-\xi_i)-(-1)^d\prod_{i=1}^d \xi_i.
$$
Thus we have~$|\det(T)|\le 2$; also obviously~$\det(T)>0$ almost surely for any odd~$d\ge 3$; however, if~$d$ is an even number, we need the following  invertibility
\begin{assumption}\label{Asu2}
If~$d$ is an even number, we assume that
$$
\prod_{i=1}^d\frac{1-\xi_i}{\xi_i}\neq 1
$$
almost surely.
\end{assumption} 

The main result of this Section is
\begin{theorem}\label{mainthm2}
Conjecture~\ref{conj} is fulfilled for all odd~$d\ge 3$, and under Assumption~\ref{Asu2} also for all even~$d\ge 4$.
\end{theorem}

From now on we will suppose that~Assumption~\ref{Asu2} is in fact fulfilled. As a result, we can always choose~$a,b\in \supp(\xi)$ such that~$a\neq b, a,b \in (0,1)$ and~$T(a_1,a_2,....,a_d)$ is invertible for all sequences~$a_1,a_2,...,a_d$ where each~$a_i\in \{a,b\}$. Let~$\mathcal{S}_{a,b}$ stand for the smallest closed semigroup which contain all of the following matrices
$$
|\det T(a_1,a_2,....,a_d)|^{-1/d}T(a_1,a_2,....,a_d),
$$ 
with~$a_1,a_2,...,a_d\in \{a,b\}$. 
We will show that~$\mathcal{S}_{a,b}\subseteq \SSS_{\mu}$ is strongly irreducible and contracting, hence so is~$\SSS_\mu$ itself. Then the result of Theorem~\ref{mainthm2} will immediately follow from Proposition~\ref{propcond} and~\ref{th343}, provided we check the condition of the latter statement (and this is done in turn in Propositions~\ref{p.st.irred} and~\ref{pr.contra} below).

\subsection{Irreducibility}
\begin{proposition}\label{p.irred}
Suppose that Assumptions~\ref{Asu1} and~\ref{Asu2} hold. Then the family of matrices 
$$
\{T(a_1,a_2,...,a_d)\}_{a_1,a_2,...,a_d\in \{a,b\}}
$$ 
is irreducible in~$\R^{d-1}$. 
\end{proposition}
\begin{proof}
Observe that, if~$W$ is a real proper invariant subspaces of linear operator~$A$ then~$\tilde{W}=\{w'+iw''\ : \ w',w''\in W\}$ is also a complex proper invariant subspaces of~$A$. Thus we can complete the proof by proving the irreducibility in~$\mathbb{C}^{d-1}$. 

From now on, let us denote
\begin{equation}\label{matrix3}
T_a=T(a,a,...,a)\text{ and }  T_{a,b;k}=T(a_1,a_2,...,a_d)|_{a_k=b, \ a_j=a, j\neq k}.
\end{equation}
Note that~$T_{a}$ has eigenvectors given by
\begin{align}\label{eigenvector}
v_1=\begin{pmatrix}
1\\ \eps\\ \eps^2\\ \vdots\\ \eps^{d-2}
\end{pmatrix},\ 
v_2=\begin{pmatrix}
1\\ \eps^2\\ \eps^4\\ \vdots\\ \eps^{2(d-2)}
\end{pmatrix},\ \dots,
v_{d-1}=\begin{pmatrix}
1\\ \eps^{d-1}\\ \eps^{(d-1)2}\\ \vdots\\ \eps^{(d-1)(d-2)}
\end{pmatrix}
\end{align}
where~$\eps=e^{2\pi i/d}$ is the~$d-$th root of~$1$; one can easily conclude that these~$d-1$ eigenvectors are linearly independent in~$\mathbb{C}^{d-1}$, and correspond to eigenvalues~$\lambda_l=1-a+a\eps^l$, $l=1,2,\dots,d-1$ respectively.

Let us prove that all complex proper invariant subspaces of~$T_{a}$ are given by the linear spans of~$2^n-2$ non-trivial subsets of~$\{v_1,\dots,v_{d-1}\}$, and only by them. First of all, suppose~$V={\rm span}(v_{k_1},v_{k_2},\dots,v_{k_m})$ where~$1\le k_1<k_2<\dots<k_m\le d-1$ and~$m\in\{1,2,\dots,d-1\}$. Since~$T_{a} v_{k_l}=\lambda_{k_l} v_{k_l}$ and~$\lambda_{k_l}\ne 0$, $1\le l\le m$, we conclude that~$\text{span}(T_{a} v_{k_1},\dots,T_{a} v_{k_m})=V$ and hence~$T_{a} (V)=V$ and thus~$V$ is indeed invariant. 

On the other hand, suppose~$V$ is an invariant subspace of~$T_{a}$, that is~$T_{a} (V)=V$.  Since~$v_1,\dots,v_{d-1}$ form a basis, any vector~$w\in V$ can be written as
$$
w=q_1 v_{k_1}+q_2 v_{k_2}+\dots+q_m v_{k_m}
$$
where all~$q_l\ne 0$. Since~$V$ is an invariant subspace, $T_{a} w\in V$, consequently
$$
w'=q_2'v_{k_2}+\dots+q_m'v_{k_m}=
q_2 (\lambda_{k_2}-\lambda_{k_1})v_{k_2}+\dots
q_m (\lambda_{k_m}-\lambda_{k_1})v_{k_m}=T_{a;a} w - \lambda_{k_1} w \in V
$$
with all~$q_l'\ne 0$ since all~$\lambda$'s are distinct.  Continuing this by induction, we will obtain that~$v_{k_m}\in V$, and hence~$v_{k_{m-1}}\in V,\dots, v_{k_1}\in V$. Therefore, $V$ contains all those~$v_k$ for which the projection of some vector~$w\in V$ on~$v_k$ has a non-zero coefficient. At the same time the span of all these~$v_k$ will contain all those vectors~$w$, hence~$V$ is the span of a subset of~$\{v_1,\dots,v_{d-1}\}$.

Next we will show that at the same time no proper invariant subspace~$V={\rm span}(v_{k_1},v_{k_2},\dots,v_{k_m})$  of~$T_{a}$ can be also an invariant subspace of~$T_{a,b;k}$, \ $k=1,2,...,d$. First, define the sequence of vectors~$u_1,\dots,u_d\in\R^{d-1}$ by
\begin{align}\label{eq_u_defined}
u_1
=\begin{pmatrix} 1\\ 0\\ 0\\ \vdots\\0\\ 0\end{pmatrix},\
u_2
=\begin{pmatrix} -1\\ 1\\ 0\\ \vdots\\0\\ 0\end{pmatrix},\
u_3
=\begin{pmatrix} 0\\ -1\\ 1\\ \vdots\\0\\ 0\end{pmatrix},\
\dots, \
u_d
=\begin{pmatrix} 0\\ 0\\ 0\\ \vdots\\-1\\ 1\end{pmatrix}.
\end{align} 
We must have~$T_{a,b;1} v_r \in V$ for all~$r\in\{k_1,k_2,\dots,k_m\}$, hence
$$
(a-b) u_1=T_{a,b;1} v_r-\lambda_r v_r\in V
$$
Now, by using the fact that
$$
(T_{a,b;k} -T_{a}) v_r=(a-b)\epsilon^{r(k-1)} u_k \in V
$$
for~$k=1,2,\dots,d$ we obtain that~$u_1,u_2,\dots,u_d\in V$.
Note that~$u_2,u_3,\dots,u_d$ are linearly independent, hence~$V={\rm span}(u_2,\dots,u_d)\equiv \R^{d-1}$. This contradiction completes the proof.
\end{proof}

\subsection{Strong irreducibility}We already know from Proposition~\ref{p.irred} that~$\mathcal{S}_{a,b}$ is irreducible. Now we aim to show its {\it strong} irreducibility.

\begin{lemma}\label{lemdirectsum}
If~$\mathcal{S}_{a,b}$  is irreducible but not strongly irreducible in~$\R^{d-1}$, then there exist proper linear subspaces~$V_1,V_2,...,V_r$ of~$\R^{d-1}$ such that
\begin{align*}
\R^{d-1}=\bigoplus_{j=1}^r V_j\text{ where }r>1, V_i\cap V_j=\{0\} \text{ if } i\ne j,
\end{align*}
where all the subspaces~$V_j$ have the same dimension, and 
$$
M(\cup_{j=1}^r V_j)=\cup_{j=1}^r V_j,
$$ 
for all~$M\in\mathcal{S}_{a,b}$.
\end{lemma}

\begin{proof}[Proof of Lemma~\ref{lemdirectsum}.]
See the remark and the equation~(2.7) on pp.~121--122 of~\cite{GG}.
\end{proof}

\vskip 5mm

\begin{proposition}\label{p.st.irred}
Suppose that Assumptions~\ref{Asu1} and~\ref{Asu2} hold. Then the semigroup~$\mathcal{S}_{a,b}$ is strongly irreducible. 
\end{proposition}

\begin{proof}

For a real linear space~$W\subset\R^{d-1}$, we define 
$$
\tilde W=\{w'+iw'',w',w''\in W\}\subset\C^{d-1},
$$ 
which is also a complex linear subspace of~$\C^{d-1}$. 

We already know that the semigroup~$\mathcal{S}_{a,b}$ is irreducible in~$\R^{d-1}$.
Suppose~$\mathcal{S}_{a,b}$ is not strongly irreducible in~$\R^{d-1}$. Then it implies from Lemma~\ref{lemdirectsum}  that there exist proper linear space~$V_1,V_2,...,V_r\subset \R^{d-1}$ such that
$$
\C^{d-1}=\bigoplus_{j=1}^r \tilde V_j,
$$ 
where~$\tilde V_j$ are disjoint linear subspaces of the same dimension, say~$m$, and
$$M(\cup_{j=1}^r \tilde V_j)=\cup_{j=1}^r \tilde V_j,$$ for all~$M\in\mathcal{S}_{a,b}$. 

The rest of the proof is organized as follows. First, we show irreducibility in the case~$m>1$. The case when~$m=1$ is split further in the sub-cases including the one where~$k=2$ and~$k\ge 3$, and yet further sub-sub-case where~$k=4$.

Observe also that from Lemma~III.4.5.b in~\cite{BL} it follows that for each~$j\in\{1,2,\dots,d-1\}$, we have~$T_a \tilde V_j =\tilde V_k$ for some~$k=k(j)$. Suppose~$k(j)\ne j$ for all~$j$. Let~$e_1,\dots,e_{d-1}$ be the basis~$\C^{d-1}$ such that~$e_1,\dots,e_m$ is the basis of~$V_1$, $e_{m+1},\dots,e_{m+m}$ is the basis of~$V_2$, etc. In this basis~$T_a$ will be a traceless matrix since all the~$V_j$ are disjoint. The property of being traceless is invariant with respect to changing the basis as~${\rm tr}(PAP^{-1})={\rm tr}(A)$. However, in the original basis~${\rm tr}(T_a)=(1-a)(d-1)-a\ne 0$ unless~$a=\frac{d-1}{d}$, but in this case we can replace~$a$ by~$b\ne a$, so we get a contradiction.

Thus we have established that~$k(j)=j$ for some~$j$; w.l.o.g.\ let us assume that~$j=1$ and consequently~$T_a V_1=V_1$. From the arguments in Proposition~\ref{p.irred} we know that~$V_1$ is a linear span of some subset of~$v_k$'s from~\eqref{eigenvector}, that is~$V_1=\spn\{w_1,\dots,w_m\}$ where 
$w_j=v_{r_j}$, for some subset~$\{r_1,\dots,r_m\}\subset \{1,2,\dots,d-1\}$. By denoting~$\eps_j:=\eps^{r_j}$, some~$d$-th root of~$1$, we get that
$w_j=\begin{pmatrix}1, \eps_j,\dots,\eps_j^{d-2}\end{pmatrix}^{\T}$. 
Let~$u_k$ be defined as in~\eqref{eq_u_defined}.
Then  
$$
T_{a,b;k} w_j= \lambda_{r_j} w_j +(a-b)\eps_j^{k-1} u_k.
$$

For every~$k$, we must have~$T_{a,b;k} V_1=V_j$ for some~$j=j(k)$.
Now suppose that there is no~$k$ such that~$T_{a,b;k} V_1=V_1$. Recall that~$V_1=\spn(w_1,\dots,w_m)$. Let
$$
V_k'=T_{a,b;k} V_1=\spn(\{\lambda_{r_j} w_j+c_k u_k,\ j=1,\dots,m\})
$$
where~$c_k=(a-b)\eps_j^{k-1}\ne 0$ for~$k=1,2,\dots,d-1-m$.
Observe that at the same time~$V_k'=V_q$ for some~$q=q(k)$, so that the collection~$V_k'$, $k=1,\dots,d-1-m$, is some subset of~$V_1,\dots,V_r$, possibly with repetitions.

Let us show that~$w_1,\dots,w_m,u_1,...,u_{d-1-m}$ are linearly independent. Indeed, to establish the rank of the matrix of~$d-1$ vectors~$w_1,\dots,w_m,u_1,u_2,\dots,u_{d-1-m}$ observe that

\begin{align*}
&\det\begin{pmatrix}
1 & 1 & \dots & 1  & 1 & -1 & 0 &\dots & 0\\
\eps_1 & \eps_2 & \dots & \eps_{m}   & 0 & 1 & -1 &\dots & 0\\
\eps_1^2 & \eps_2^2 & \dots & \eps_{m}^2   & 0 & 0 & 1 &\dots & 0\\
\vdots & \vdots & \ddots & \vdots   & \vdots & \vdots & \vdots &\ddots & \vdots \\
\eps_1^{d-m-2} & \eps_2^{d-m-2} & \dots & \eps_{m}^{d-m-2} 
& 0 & 0 & 0 &\dots & 1\\
\vdots & \vdots & \ddots & \vdots   & \vdots & \vdots & \vdots &\ddots & \vdots \\
\eps_1^{d-2} & \eps_2^{d-2} & \dots & \eps_{m}^{d-2} 
& 0 & 0 & 0 &\dots & 0
\end{pmatrix} & \\ &
=\det
\begin{pmatrix}
\eps_1^{d-m-1} & \eps_2^{d-m-1} & \dots & \eps_{m}^{d-m-1} 
\\
\vdots & \vdots & \ddots &  \vdots \\
\eps_1^{d-2} & \eps_2^{d-2} & \dots & \eps_{m}^{d-2} 
\end{pmatrix}
\\
&=\eps_1^{d-m-1}\cdot\dots\cdot \eps_{m}^{d-m-1} 
\cdot \det
\begin{pmatrix}
1 & 1 & \dots & 1
\\
\vdots & \vdots & \ddots &  \vdots \\
\eps_1^{m-1} & \eps_2^{m-1} & \dots & \eps_{m}^{m-1} 
\end{pmatrix} & \\ & 
=\prod_{j=1}^m \eps_j^{d-m-1}
\cdot \prod_{1\le j<k\le m} (\eps_j-\eps_k)\ne 0
\end{align*}
since this is a Vandermonde matrix. This, in turn, implies that the subspaces~$V_1,V_1',V_2',\dots,V_{d-m-1}'$ are all pairwise distinct; otherwise there would be a vector which at the same time belongs to~$\spn(\{\lambda_{r_j} w_j+c_k u_k,\ j=1,\dots,m\})$ and~$\spn(\{\lambda_{r_j} w_j+c_l u_l,\ j=1,\dots,m\})$ for~$k\ne l$, yielding linear dependence for the set~$w_1,\dots,w_m,u_k,u_l$ which is impossible.

On the other hand, it implies that the dimension of~$V_1\oplus V_1'\oplus\cdots \oplus V_{d-m}'$ is~$m\times (d-1-m)>d-1$ unless~$m=1$, yielding a contradiction that this is a subspace of~$\R^{d-1}$.

\vskip 5mm

Thus now we have to deal only with the case~$m=1$.
In this case, all the spaces~$V_1,V_2,\dots,V_{d-1}$ are one-dimensional, moreover, by letting~$\nu=\eps_1$
\begin{align*}
w_1&=(1,\nu,\dots,\nu^{d-2})^{\T}  ,\\
V_1&=\spn(w_1),\\
V_k'&:=T_{a,b;k} V_1=\spn(\lambda_{r_1}w_1+c_k u_k),\ k=1,2,\dots,d-1,
\end{align*}
and~$V_k'$s are some subset of~$V_2,\dots,V_{d-1}$ (if~$V_k'=V_1$ for some~$k$ then~$u_k\in\spn(w_1)$ which is impossible for~$d\ge 4$). If all the elements of the set~$V_1,V_1',\dots,V_{d-1}'$ are distinct (we know that then they must be linearly independent since~$\R^{d-1}=V_1\oplus V_2\oplus\cdots\oplus V_{d-1}$) this would yield a contradiction as our space is only~$(d-1)$-dimensional.

Observe that 
\begin{align*}
\det(w_1,u_2,u_3,\dots,u_{d-1})
 & =  \det\begin{pmatrix}
1      &-1 & 0 &\dots & 0\\
\nu   & 1 & -1 &\dots & 0\\
\nu^2 & 0 &  1 &\dots & 0\\
\vdots & \vdots & \vdots &\ddots & \vdots\\
\nu^{d-2} & 0 &  0 &\dots & -1 \\
\nu^{d-1} & 0 &  0 &\dots & 1
\end{pmatrix}\\
& = 1+\nu+\dots+\nu^{d-2}=\frac{1-\nu^{d-1}}{1-\nu}=\frac{-1}{\nu}\ne 0
\end{align*}
since~$\nu^d\equiv \eps_1^d=1$. This implies that the vectors~$w_1,u_2,u_3,\dots,u_{d-1}$ are linearly independent and hence it is impossible that~$V_k'=V_h'$ for some~$k,h\in\{2,\dots,d-1\}$ such that~$k\ne h$. 

\vskip 5mm

So the only not covered case is when~$V_1'$ coincides with some~$V_k'$, $k=2,\dots,d-1$, implying a linear dependence between~$w_1$, $u_1$ and~$u_k$. However, if~$k=2$, then
\begin{align*}
\rank(w_1,u_1,u_k)&=
\rank\begin{pmatrix}
1  & \nu & \nu^2 & \dots &\nu^{d-2}\\
1  & 0 & 0 & \dots&0\\
-1  & 1 & 0 & \dots&0
\end{pmatrix}
\\
&=1+
\rank\begin{pmatrix}
 \nu & \nu^2 & \dots &\nu^{d-2}\\
 1 & 0 & \dots & 0
\end{pmatrix}=3
\end{align*}
since~$\nu^2\ne 0$. Finally, if~$k\ge 3$, then
\begin{align*}
\rank(w_1,u_1,u_k)&=
\rank\begin{pmatrix}
1  & \nu &\dots& \nu^{k-2} & \nu^{k-1} & \nu^{k} & \dots&\nu^{d-2}\\
1  & 0&\dots& 0 & 0 & 0 &\dots&0\\
0  & 0&\dots& -1 & 1 & 0 & \dots&0
\end{pmatrix}\\
&=1+
\rank\begin{pmatrix}
\nu &\dots& \nu^{k-2} & \nu^{k-1} & \nu^{k} & \dots&\nu^{d-2}\\
0&\dots& -1 & 1 & 0 & \dots &0
\end{pmatrix}=3
\end{align*}
unless simultaneously~$d=4$, $k=3$ and~$\nu=\eps_1=-1$.

\vskip 5mm

Finally, to deal with the case~$d=4$ and~$\eps_1=-1$, observe that 
$$
T(\xi_1,\xi_2,\xi_3,\xi_4)=\begin{pmatrix}
1-\xi_1 &\xi_2 & 0\\
0& 1-\xi_2 &\xi_3 \\
-\xi_4 &-\xi_4  &1-\xi_3 -\xi_4 
\end{pmatrix},
\quad 
w_1=\begin{pmatrix}
1 \\ -1\\ 1
\end{pmatrix}=e_1-e_2+e_3
$$
where~$e_1,e_2,e_3$ are the standard basis vectors for~$\R^3$.
Let us consider
\begin{align*}
w_1^*:=T(a,a,b,a) \, w_1&=(1-b-a) w_1+(b-a)e_1,
\\
w_2^*:=T(a,b,b,a) \, w_1&=(1-b-a) w_1+(b-a)e_2,
\\
w_3^*:=T(a,b,a,a) \, w_1&=(1-b-a) w_1+(b-a)e_3.
\end{align*}
Then, in the standard Euclidean coordinates,
$$
A:=[w_1^*,w_2^*,w_3^*]=\begin{pmatrix}
1-2 a& b+a-1& 1-b-a\\
1-b-a&2 b-1&1-b-a\\
1-b-a&b+a-1&1-2 a
\end{pmatrix},\ \text{and }\det(A)=(b-a)^2(1-2a).
$$
From Assumptions~\ref{Asu1} and~\ref{Asu2} it follows that w.l.o.g.\ we can chooses~$a$ and~$b$ such that~$a\ne 1/2$, $a\ne b$, and~$a+b\ne 1$, implying that the above determinant is non-zero. Thus we obtain that the three subspaces span by~$w_1^*$, $w_2^*$, $w_3^*$ are linearly independent in~$\R^3$ again yielding a contradiction.
\end{proof}

\subsection{Contracting property}

Here we need to show that the semigroup~$\mathcal{S}_{a,b}$ is strongly irreducible and contracting. While in general it is not easy to verify the contraction property of a semigroup, thanks to the following important statement by Goldsheid and~Margulis in~\cite{GM}, it suffices to check this property for the Zariski closure of~$\mathcal{S}_{a,b}$ (which is easier).

\begin{definition}
Zariski closure of a subset~$H$ of an algebraic manifold is the smallest algebraic submanifold that contains~$H$.
\end{definition}
\begin{proposition}[Lemma 3.3 in~\cite{GM}]
\label{propGM1}
The Zariski closure~${\sf Zr}(H)$ of a closed semigroup of~$H\subset GL(d,\R)$ is a group.
\end{proposition}

\begin{proposition}[Lemma 6.3 in~\cite{GM}]
\label{propGM2}
If a closed semigroup~$H\subset GL(d,\R)$ is strongly irreducible and its Zariski closure~${\sf Zr}(H)$ has the contraction property then~$H$ also has the contraction property. 
\end{proposition}

\begin{proposition}\label{pr.contra}
Suppose that Assumptions~\ref{Asu1} and~\ref{Asu2} hold. Then the semigroup~$\mathcal{S}_{a,b}$ is contracting.
\end{proposition}

\begin{proof}
According to Proposition~\ref{propGM2}  it is sufficient to show that~${\sf Zr}(\mathcal{S}_{a,b})$ is contracting, since we have already established that~$\mathcal{S}_{a,b}$ and hence~${\sf Zr}(\mathcal{S}_{a,b})$ is strongly irreducible by Proposition~\ref{p.st.irred}.
Note that~$T^{-1}\in {\sf Zr}(\mathcal{S}_{a,b})$ for any~$T\in \mathcal{S}_{a,b}$, since the Zariski closure is necessary a group by Proposition~\ref{propGM1}. We consider two separate cases.

\vskip 5mm \noindent
\textbf{Case~$d=2l+1$ is odd.} Define
\begin{align*}
M&=T(a,b,\dots,a,b,{\bf a}) \, 
T(a,b,\dots,a,b,{\bf b})^{-1}
T(b,a,\dots,b,a,{\bf b}) \, 
T(b,a,\dots,b,a,{\bf a})^{-1}\in {\sf Zr}(\mathcal{S}_{a,b})
\end{align*}
After some algebraic computations, one can obtain that
\begin{align*}M=\begin{pmatrix}
1 & 0 & \dots & 0 & 0\\
0 & 1 & \dots & 0 & 0\\
\vdots & \vdots & \ddots & \vdots & \vdots \\
0 & 0 & \dots & 1 & 0\\
\phi_1 & \phi_2 & \dots & \phi_{2l-1} & 1
\end{pmatrix},
\end{align*}
where 
$$
\phi_{2j-1}=-\frac{(a-b)^2\left((1-a)(1-b)\right)^{l-j}(ab)^{j-1}}{(1-a)^l(1-b)^{l+1}+a^lb^{l+1}}, \ \text{and}\ \phi_{2j}=0,\ \ j=1,2,...,l.
$$
Hence
\begin{align*}M^n=\begin{pmatrix}
1 & 0 & \dots & 0 & 0\\
0 & 1 & \dots & 0 & 0\\
\vdots & \vdots & \ddots & \vdots & \vdots \\
0 & 0 & \dots & 1 & 0\\
n\phi_1 & n\phi_2 & \dots & n\phi_{2l-1} & 1
\end{pmatrix}\in {\sf Zr}(\mathcal{S}_{a,b}).
\end{align*}
It implies that~$\|M^n\|\approx Const \cdot n$ hence~$||M^n||^{-1} M^n$ converges to a matrix whose first~$d-2$ rows are zero rows, and thus~${\sf Zr}(\mathcal{S}_{a,b})$ is contracting by definition.

\vskip 5mm \noindent
\textbf{Case~$d=2l$ is even}. Define
\begin{align*}
M=T(a,a,\dots,a,a,a,{\bf a}) \, 
T(a,a,\dots,a,a,a,{\bf b})^{-1}\\
=\begin{pmatrix}
1 & 0 & \dots & 0 & 0\\
0 & 1 & \dots & 0 & 0\\
\vdots & \vdots & \ddots & \vdots & \vdots \\
0 & 0 & \dots & 1 & 0\\
c_1 & c_2 & \dots & c_{d-1} & c(a,b)
\end{pmatrix}
\end{align*}
where~$c_1=c_1(a,b),\dots,c_{d-1}=c_{d-1}(a,b)$
are some constants depending on~$a$ and~$b$, and~$c(a,b)=\det T(a,\dots,a,a)/\det T(a,\dots,a,b)$; observe also that
\begin{align*}
\det T(a,\dots,a,a)&=(1-a)^d-a^d\\
\det T(a,\dots,a,b)&=(1-a)^d-a^d+(a-b)[(1-a)^{d-1}+a^{d-1}]
\end{align*}

Assume initially that~$|c(a,b)|>1$, then 
$$
M^n=\begin{pmatrix}
1 & 0 & \dots & 0 & 0\\
0 & 1 & \dots & 0 & 0\\
\vdots & \vdots & \ddots & \vdots & \vdots \\
0 & 0 & \dots & 1 & 0\\
A_n c_1 & A_n c_2 & \dots & A_n c_{n-1} & c(a,b)^n
\end{pmatrix}
$$
where~$A_n=1+c(a,b)+c(a,b)^2+...+c(a,b)^{n-1}$,
so that~$||M^n||\ge {\sf const} \times c(a,b)^n\to\infty$ and thus~$||M^n||^{-1} M^n$ converges to a matrix whose first~$d-2$ rows are zeros. If~$|c(a,b)|<1$ then we can consider~$M^{-1}$ instead of~$M$, which has the form 
$$
M^{-1}=\begin{pmatrix}
1 & 0 & \dots &  0\\
\vdots & \vdots & \ddots &  \vdots \\
* & * & \dots  & c(a,b)^{-1}
\end{pmatrix}\in {\sf Zr}(\mathcal{S}_{a,b})
$$
and then apply exactly the same arguments as when~$|c(a,b)|>1$. Note that~$c(a,b)\neq 1$ since~$a\neq b$, so we only have to consider the case when~$c(a,b)=-1$. 
 
We have~$c(a,b)\neq c(b,a)$ since~$a\ne b$. Hence, w.l.o.g.\ we can assume that~$c(a,b)\neq -1$. 
So in all the cases, either~$||M^{n}||^{-1}M^{n}$ or~$||M^{-n}||^{-1}M^{-n}$ converges to a rank one matrix as~$n\to \infty$. 
\end{proof}

\section{Convergence rate of random polygons}
\label{Secrate}

\subsection{Convergence rate of rescaled polygons to flatness}
Throughout the rest of the paper we use the notation~$\log^+(x):=\max\{\log x, 0\}$.

Let~$\ell(T)=\max(\log^+(||T||),\log^+(||T^{-1}||))$. In this section, we suppose that Assumptions~\ref{Asu1} and~\ref{Asu2} as well as the following condition hold
\begin{align}\label{eqlogpm}
\E \ell(T) <\infty.
\end{align}

Let~$T_1,T_2,....$ be a sequence of random matrices having the same distribution as~$T$. We define Lyapunov exponents
\begin{align}\label{eqLyapexp}
\mu_j=\lim_{n\to\infty}\E\left(\frac{1}{n}\log~\sigma_j^{(n)}\right),\quad  j=1,2,...,d-1
\end{align}
where~$\sigma_1^{(n)}\ge \sigma_2^{(n)} ...\ge \sigma_{d-1}^{(n)}$ are the singular values of~$T^{(n)}=T_n T_{n-1}\dots T_1$, i.e., the square roots of the eigenvalues of~$\left(T^{(n)}\right)^{\T}  T^{(n)}$. Therefore, from the proof of Proposition~III.6.4 in~\cite{BL} (pp.~67--68), 
for any~$x,y\in{P}(\R^{d-1})$
\begin{align}\label{eq.exp.conv}
\lim_{n\to\infty} \frac{1}{n}\log \delta(T^{(n)} x, T^{(n)} y)\le \mu_2-\mu_1<0 \quad \text{ a.s.}
\end{align}

\begin{lemma}\label{lemheadtail}
Let~$\xi_1,\xi_2,\dots,\xi_d\in [0,1]$. Then
\begin{align*}
\prod_{i=1}^d \xi_i(1-\xi_i)\le  \xi_1 \xi_2\dots \xi_d+(1-\xi_1)(1-\xi_2)\dots (1-\xi_d)\le 1.
\end{align*}
\end{lemma}
\begin{proof}
The upper bound follows from the fact that it is equal to probability to get either all heads or all tails in an experiment with throwing~$d$ independent coins each with probability to turn up head equal to~$\xi_i$, $i=1,2,\dots,d$. To get the lower bound observe that for~$d=1,2,\dots$ we have
$$
\prod_{i=1}^{d}\xi_i +\prod_{i=1}^{d}(1-\xi_i)
\ge 
\left[
\prod_{i=1}^{d-1}\xi_i +\prod_{i=1}^{d-1}(1-\xi_i)\right]\cdot \xi_d(1-\xi_d)
$$
and since the statement is true for~$d=1$, we have proved the proposition.
\end{proof}

As it is implied from the following proposition, we can reformulate the requirement~\eqref{eqlogpm} as 
\begin{assumption}\label{Asu3}
\begin{align*}
\E  \log\left(|\det(T)|\right)=\E \log \left|\prod_{i=1}^d (1-\xi_i)-(-1)^d\prod_{i=1}^d \xi_i\right| >-\infty.
\end{align*}
\end{assumption}

\begin{proposition}\label{propAsu3}
Condition~\eqref{eqlogpm} holds if and only if
Assumption~\ref{Asu3} is fulfilled.
\end{proposition}
\begin{proof}
Noticing that all the elements of~$T$ are bounded,  and using the formula for inversion of matrices we obtain that
\begin{align}\label{eqT-1<det}
||T||\le C_1, \quad ||T^{-1}||\le \frac{C_2}{|\det(T)|}
\end{align}
where~$C_i$, $i=1,2,\dots$ here and further in the text denote some non-random positive constants.
Let~$\sigma_1\ge \sigma_2\ge \dots \sigma_{d-1}>0$ be the singular values of matrix~$T$, that is, the square roots of the eigenvalues of~$T^{\T} T$, arranged in the decreasing order. Then~$||T^{-1}||=1/\sigma_{d-1}$. On the other hand, using the fact that there is a unitary matrix~$U$ such that~$U^{\T}  (T^{\T}   T)U$ is a diagonal matrix with elements~$\sigma_i^2$, we obtain that 
$$
\det(T)=\sigma_1 \sigma_2 \dots \sigma_{d-1}
\ge \left(\sigma_{d-1}\right)^{d-1}
$$
so that
$$
||T^{-1}||=\frac 1{\sigma_{d-1}}\ge 
\frac{1}{\left|\det(T)\right|^{\frac{1}{d-1}}}.
$$
On the other hand it is easy  that
$$
\det(T)=\prod_{i=1}^d (1-\xi_i)-(-1)^d\prod_{i=1}^d \xi_i
$$
which is always non-negative for odd~$d$, but can be  positive as well as negative for even~$d$; in  both cases~$|\det(T)|\le 1$, as it easily follows from Lemma~\ref{lemheadtail}. Consequently,
\begin{align*}
\log^+ \left(||T^{-1}|| \right)
&\le
\log^+\left(\frac{C_2}{|\det(T)|}\right)
\le 
\log^+\left(\frac{C_2+1}{|\det(T)|}\right)
\\ & \le 
\log\left(\frac{1}{|\det(T)|}\right)
 =
  -\log\left(|\det(T)|\right),
\\
\log^+\left( ||T^{-1}|| \right) 
&\ge 
\log^+\left(\frac{1}{|\det(T)|^{\frac{1}{d-1}}}\right)
\ge -\frac 1{d-1} \log\left(|\det(T)|\right).
\end{align*}
Since~$\log^+ ||T||$ is bounded above by some constant,  the statement of the proposition follows.
\end{proof}

Notice that since
\begin{align}\label{eqsumofsv}
\mu_1+\mu_2+...+\mu_{d-1}=\E(\log|\det(T)|)
\end{align}
all Lyapunov exponents~$\mu_j$, $j=1,2,...,d-1$ are finite if and only if Assumption~\ref{Asu3} is fulfilled. Therefore, using~\eqref{eq.exp.conv}, we can deduce the following
\begin{theorem}\label{mainthm3}
Suppose that Assumptions~\ref{Asu1}, \ref{Asu2} and~\ref{Asu3} hold. Then the sequence of polygons~$L_n$ converges to flatness with at least exponential rate with parameter~$\mu=\mu_1-\mu_2\in (0,\infty)$
\end{theorem}

Now let us give an  ``easier" sufficient condition for Assumption~\ref{Asu3} which  depends only on the distribution of one~$\xi$. 

\begin{proposition}\label{propsuff}
Suppose that~$d=3,5,\dots$ is odd. If~$\E|\log \xi| <\infty$
and~$\E|\log (1-\xi)| <\infty$ then Assumption~\ref{Asu3} is fulfilled. A sufficient condition for these expectations to be finite is 
\begin{equation}\label{discond}
\limsup_{v\downarrow 0} \frac{\P(\xi<v)}{v^\alpha}<\infty
\text{ and }
\limsup_{v\uparrow 1} \frac{\P(\xi>v)}{(1-v)^\alpha}<\infty
\end{equation}
for some~$\alpha>0$.
\end{proposition}

\begin{remark}
Note that when~$d$ is even we would not be able to bound~$|\det(T)|$ from below by the products of~$\xi_i(1-\xi_i)$ as easily as it is done in the following proof. Indeed, if we let all~$\xi_i=1/2$ then~$\det(T)=0$ while all~$\xi_i(1-\xi_i)=1/4>0$.
\end{remark}

\begin{proof}[Proof of Proposition~\ref{propsuff}.]
The first part of the statement follows immediately from Lemma~\ref{lemheadtail}  since
\begin{align*}
\E\log |\det (T)|&=\E\log\left[\prod_{i=1}^{d}\xi_i +\prod_{i=1}^{d}(1-\xi_i)\right]\\
& \ge \E\log\left[\prod_{i=1}^{d}\xi(1-\xi_i)\right]
=\sum _{i=1}^{d}\left(\E \log \xi_i+\E \log (1-\xi_i)\right).
\end{align*}
To prove the second part, note that
\begin{align*}
\E|\log\xi|&\le 1+\E\left[|\log\xi|\cdot 1_{\xi<e^{-1}}\right]
=1+\int_0^{\infty} \P\left(-(\log \xi) \cdot 1_{\xi<e^{-1}}>u\right)\rmd u
\\ & =1+\int_0^1\dots+\int_1^{\infty}\dots\\
&=1+\int_0^1 \P\left(e \xi  <1\right)\rmd u
+\int_1^{\infty} \P\left(-\log \xi >u\right)\rmd u
\\ &=1+\P\left(e \xi <1\right)+\int_{0}^{e^{-1}} \frac{\P\left(\xi<v\right)}{v}\rmd v<\infty
\end{align*}
since 
$$\frac{\P\left(\xi<v\right)}{v}\le \frac{\rm const}{v^{1+\alpha}}$$ 
for sufficiently small~$v$.
The expectation~$\E|\log(1-\xi)|$ is bounded in exactly the same way.
\end{proof}

An interesting example is when~$\xi$ has a uniform distribution, as in the paper~\cite{VOL}.

\begin{proposition}
If the distribution of~$\xi$ is uniform on~$[0,1]$ then Assumption~\ref{Asu3} is fulfilled for all~$d\ge 3$.
\end{proposition}
\begin{proof}
The case when~$d$ is odd immediately follows from Proposition~\ref{propsuff} so we assume that~$d$ is even. We have
\begin{align*}
&\E \log |\det T|  =\int_0^1\dots\int_0^1
\log\left|(1-x_1)\dots(1-x_d)-x_1\dots x_d\right| \rmd x_1\dots \rmd x_d
\\&=
\int_0^1\dots\int_0^1
\log\left(x_1\dots x_d\right) \rmd x_1\dots \rmd x_d
\\ & +
\int_0^1\dots\int_0^1
\log\left|1-\frac{1-x_1}{x_1}\dots\frac{1-x_d}{x_d}\right| \rmd x_1\dots \rmd x_d \\
&=
-d+ \int_0^\infty \dots\int_0^\infty
\frac{\log\left|1-u_1\dots u_d\right|}
{(1+u_1)^2\dots (1+u_d)^2}\rmd u_1\dots \rmd u_d 
\\ &=
-d+ \int_0^\infty \dots\int_0^\infty
\frac{u_1\dots u_{d-1}}
{(1+u_1)^2\dots (1+u_{d-1})^2}
\left(\int_0^\infty
\frac{\log|1-v| \rmd v}
{ (u_1\dots u_{d-1}+v)^2}
\right)\rmd u_1\dots \rmd u_{d-1} 
\end{align*}
where the inner integral
\begin{align*}
&\int_0^\infty
\frac{\log|1-v| \rmd v}
{ (u_1\dots u_{d-1}+v)^2}  =\left(\int_0^{1/2}+\int_{1/2}^{3/2}
+\int_{3/2}^{2}+\int_{2}^{\infty}\right)
\frac{\log|1-v|}
{(u_1\dots u_{d-1}+v)^2}
\rmd v
\\ &\ge 
\int_0^{1/2}
\frac{-\log 2}
{(u_1\dots u_{d-1}+v)^2}
\rmd v
 +
\int_{1/2}^{3/2}
\frac{\log |1-v|}
{(u_1\dots u_{d-1}+1/2)^2}
\rmd v \\ & +
\int_{3/2}^{2}
\frac{-\log 2}
{(u_1\dots u_{d-1}+v)^2}
\rmd v
+
0\\
&\ge 
\int_0^{\infty}
\frac{-\log 2}
{(u_1\dots u_{d-1}+v)^2}
\rmd v
+
\int_{1/2}^{3/2}
\frac{\log |1-v|}
{(u_1\dots u_{d-1}+1/2)^2}
\rmd v
\\& = -\frac{\log 2}{u_1\dots u_d}+
-\frac{1+\log 2}{(u_1\dots u_{d-1}+1/2)^2}.
\end{align*}
Consequently,
\begin{align*}
&\E \log |\det T|\ge  -d -\log 2\,
\int_0^\infty \dots\int_0^\infty
\frac{\rmd u_1\dots \rmd u_{d-1} }
{(1+u_1)^2\dots (1+u_{d-1})^2}\\
&-
\int_0^\infty \dots\int_0^\infty
\frac{(1+\log 2)[u_1\dots u_{d-1}]\rmd u_1\dots \rmd u_{d-1} }
{(1+u_1)^2\dots (1+u_{d-1})^2(1/2+[u_1\dots u_{d-1}])^2}
\\
&\ge  -d -\left[\log 2 +\frac{1+\log 2}{2}\right]
\int_0^\infty \dots\int_0^\infty
\frac{\rmd u_1\dots \rmd u_{d-1} }
{(1+u_1)^2\dots (1+u_{d-1})^2}> -\infty
\end{align*}
since~$a/(1/2+a)^2\le 1/2$ for~$a\ge 0$.
\end{proof}

The next statement shows that there are, in fact, examples of distributions for which Assumption~\ref{Asu3} is {\em not fulfilled}.
\begin{proposition}
Suppose~$\xi_i$ are i.i.d.\ with density
$$
f(x)=\begin{cases}  \displaystyle
\frac{c}{x \log^{1+\delta} x}, & 0<x\le 1/2;\\
\displaystyle \frac{c}{(1-x) \log^{1+\delta} (1-x)}, & 1/2<x<1;\\
0, &\text{otherwise}
\end{cases}
$$
where~$\delta\in(0,1/2]$  and~$c=c(\delta)\in(0,\infty)$ is the appropriate constant. Then Assumption~\ref{Asu3} is \underline{not} satisfied.
\end{proposition}
\begin{proof}
Assuming~$d$ is odd  and noticing that~$f(1-y)=f(y)$ and that 
$$x_1\dots x_d+(1-x_1)\dots(1-x_d)\le 1$$ 
by Lemma~\ref{lemheadtail} we have
\begin{align*}
& \E \log |\det T|=\int_0^1\dots\int_0^1
\log\left(x_1\dots x_d+(1-x_1)\dots(1-x_d)\right) f(x_1)\dots f(x_d)\rmd x_1\dots \rmd x_d
\\&\le
\int_0^1\dots\int_0^1
\log\left(x_1 x_2+(1-x_1)(1-x_2)\right)
f(x_1)\dots f(x_d)\rmd x_1\dots \rmd x_d\\
&=\int_0^1\int_0^1 \log(x(1-y)+y(1-x))f(x)f(y)\rmd x \rmd y\\
&\le \int_0^{1/2}\int_0^{1/2} \log(x+y-xy)f(x)f(y)\rmd x \rmd y \\ &
\le \int_0^{1/2}\int_0^{1/2} \log(x+y)f(x)f(y)\rmd x \rmd y = \int_0^{1/2}\int_0^{1/2} 
\frac{\log(x+y)}{(x\log^{1+\delta} x)( y\log^{1+\delta} y) }\rmd x \rmd y \\&
=
\int_{\log 2}^{\infty}\int_{\log 2}^{\infty} 
\frac{\log(e^{-u}+e^{-v})}{u^{1+\delta} v^{1+\delta}}\rmd u \rmd v =
2\int_{\log 2}^{\infty}\int_{\log 2}^{\infty} 
\frac{\log(e^{-u}+e^{-v})}{u^{1+\delta} v^{1+\delta}} 1_{u>v}\rmd u \rmd v \\ &
\le
2\int_{\log 2}^{\infty}
\left(
\int_{\log 2}^{\infty} 
\frac{\log(2e^{-v})}{u^{1+\delta} v^{1+\delta}} 1_{u>v}\rmd u\right) \rmd v =
\frac{2}{\delta}\int_{\log 2}^{\infty}
\frac{\log(2)-v}{ v^{1+2\delta}}  \rmd v=-\infty
\end{align*}
since~$\delta\le 1/2$. The case when~$d$ is even can handled similarly.
\end{proof}

The next statement shows how quickly the lengths of the largest side of the polygon converge to zero.
\begin{lemma}\label{lemtildeS}
Suppose that Assumptions~\ref{Asu1}, \ref{Asu2}, \ref{Asu3} are fulfilled. Let 
\begin{align}\label{eqMndefined}
\M_n=\max_{j=1,\dots,d}\| l_j^{(n)}\|
\end{align}
be the length of the largest side of~$L_n$. Then
$$
\lim_{n\to\infty}\frac{1}{n}\log (\M_n)= \mu_1 \quad \text{a.s.}
$$
\end{lemma}
\begin{proof}
First of all, observe that by the triangle inequality
$$
\max_{j=1,\dots,d-1}\| l_j\|
\le \max_{j=1,\dots,d}\|l_j\|
\le \max\left\{\| l_1\|+\dots+ \| l_{d-1}\|,\max_{j=1,\dots,d-1}\| l_j\|,\right\}\le (d-1)\max_{j=1,\dots,d-1} \| l_j\|
$$
so it suffices to prove the statement of the lemma for the first~$d-1$ sides of~$L_n$, i.e., we can redefine just inside of this proof~$\M_n$ as~$\max\{\|l_1^{(n)}\|,\| l_2^{(n)}\|,\dots,\| l_{d-1}^{(n)}\|\}$. Also, to avoid confusion, in this proof we will denote by~$\|\cdot\|_{(k)}$ the Euclidean norm in~$\R^k$, while~$\|\cdot\|$ is just a Euclidean norm in~$\R^2$.
By applying Theorem~III.7.2.i (pp.~72) in~\cite{BL},  we obtain
\begin{align}\label{eqII72}
\lim_{n\to\infty}\frac{1}{n}\log\| T^{(n)} x\|_{(d-1)}=\lim_{n\to\infty}\frac{1}{n}\log\| T_n...T_2T_1x\|_{(d-1)}=\mu_1
\quad \text{for each}\quad x\in\R^{d-1}\setminus\{0\}.
\end{align}
Now recall that the coordinates of~$l_j^{(n)}\in\R^2$ are the~$j$-th coordinates of~$x^{(n)}=T^{(n)} x^{(0)}$ and~$y^{(n)}=T^{(n)} y^{(0)}$ respectively. Omitting the superscript~${}^{(n)}$, we have
$$
\|l_j\|^2=x_j^2+y_j^2, \quad
\|x\|_{(d-1)}^2=x_1^2+\dots+x_{d-1}^2,
, \quad
\|y\|_{(d-1)}^2=y_1^2+\dots+y_{d-1}^2,
$$
so
$$
\frac{\|x\|^2_{(d-1)}}{d-1}\le \max_{j=1,\dots,d-1} x_j^2 \le \max_{j=1,\dots,d-1} \|l_j\|^2
\le x_1^2+\dots+x_{d-1}^2
+y_1^2+\dots+y_{d-1}^2=\|x\|_{(d-1)}^2+\|y\|_{(d-1)}^2
$$
Together with~\eqref{eqII72} this immediately implies
$$
\limsup_{n\to\infty}\frac{1}{n}\log \max\{\|l_1^{(n)}\|,\|l_2^{(n)}\|,\dots,\|l_{d-1}^{(n)}\|\} =\mu_1.
$$
\end{proof}

\subsection{Convergence rate of polygon vertices}
The purpose of this Section is to calculate the exact speed of convergence of (not rescaled) polygons~$L_n$ to a random point in the plane in the general case~$d\ge 3$, under some conditions.

Let~$(a_j^{(n)} , b_j^{(n)}),\ j=1,2,...,d$, be the  Cartesian coordinates of vertices~$A_d^{(n)},A_1^{(n)},A_2^{(n)},...,A_{d-1}^{(n)}$ respectively -- please note the unusual enumeration of the coordinates, which we do in order to use the same notation for matrix~$H$ given by~\eqref{matrix1}. We have the following linear relation
$$
a^{(n)}  = H_{n+1}^{\T}   a^{(n-1)} , b^{(n)}  = H_{n+1}^{\T}   b^{(n-1)}
$$
where~$a^{(n)}=\left(
a_1^{(n)},a_2^{(n)},...,a_{d}^{(n)}\right),b^{(n)}
=\left(b_1^{(n)},b_2^{(n)},...,b_{d}^{(n)}\right)$. 
We will make use of the following
\begin{proposition}[Theorem~4 in~\cite{MCK}]\label{stoch}
Let~$(X_k)_{k\ge 1}$ be a sequence of i.i.d.\ random stochastic~$d\times d$ matrices such that~$X_{n_0}X_{n-1}...X_2X_1$ is a positive matrix with a positive probability for some~$n_0<\infty$. Then there exists a random nonegative vector~$W=(w_1,w_2,...,w_d)$ such that~$w_1+w_2+...+w_d=1$ and
$$
X_nX_{n-1}...X_2X_{1}\to {\bf 1}^{\T}   W
$$
almost surely as~$n\to\infty$, where~${\bf 1}=(1,1,...,1)$. Moreover, if~$V=(v_1,...,v_d)$ is a random non-negative vector such that~$v_1+v_2+...v_d=1$, $V$ is independent of~$X_1$ then~$VX_1=V$ in distribution if and only if~$V=W$ in distribution. 
\end{proposition}
\begin{theorem}\label{mainthm5}
Suppose that Assumptions~\ref{Asu1} and~\ref{Asu3} hold then  the polygon~$L_n$ converges almost surely to a random point~$P$ inside the initial polygon~$L_0$ such that
$$
\max_{j=1,2,...,d} \| PA^{(n)}_j\| \sim C e^{\mu_1n}
$$ 
almost surely as~$n\to \infty$, in the sense that~$\frac 1n \, \log \left(\max \| PA^{(n)}_j\|\right) \to \mu_1$ where~$\mu_1$ is defined in~\eqref{eqLyapexp}.
\end{theorem}
\begin{proof}
By Assumption~\ref{Asu1} we have that~$H_d^{\T}   H_{d-1}^{\T}  ...H_2^{\T}  H_{1}^{\T}$ is almost surely a positive stochastic matrix. Therefore, from Proposition~\ref{stoch} it follows that there exists a random non-negative vector~$\zeta=(\zeta_1,...,\zeta_d)$ such that~$\zeta_1+...+\zeta_d=1$ for which
$$
a^{(n)} \to \left(\zeta_1 a^{(0)}_1+...+\zeta_d a^{(0)}_d\right){\bf 1}
$$
and 
$$
b^{(n)} \to \left(\zeta_1 b^{(0)}_1+...+\zeta_d b^{(0)}_d\right){\bf 1}
$$
almost surely as~$n\to\infty$. It implies that the sequence of polygon~$L_n$ converges to a random point~$P$ defined by the following vector identity
$$
OP=\zeta_1 OA_d^{(0)}+\zeta_2 OA_1^{(0)}+...+\zeta_d OA_{d-1}^{(0)}
$$
where~$O=(0,0)$ is the origin of the Cartesian plane. (Observe that if~$\xi_i$ is~$\text{Beta}(\alpha,\beta)$ distributed on~$(0,1)$ then~$\zeta=(\zeta_1,...,\zeta_d)$ is a Dirichlet distributed random vector with parameters~$(\alpha+\beta,\alpha+\beta,...,\alpha+\beta)$.~)

Since
$$
\| PA_j\|< \|A_dA_1 \|+\|A_1A_2 \| +...+\|A_{d-1}A_d \|\le d \times \max_{k=1,2,...,d}\|A_kA_{k+1}\|.
$$
and on the other hand, for each~$k=1,2,...,d$, we have
$$ 
\max_{j=1,2,...,d} \| PA_j\| 
\ge \frac{1}{2} \left(\| PA_k\|+
\| PA_{k+1}\|\right)
\ge \frac{1}{2}  \| A_kA_{k+1}\|
$$
the following inequality inequalities hold:
\begin{align}
\label{lenght1}
\M_n
\le \max_{j=1,2,...,d} \| PA^{(n)}_j\|\le d \,
\M_n.
\end{align} 
Under the Assumption~\ref{Asu3} we have (see our Lemma~\ref{lemtildeS} and Proposition~III.7.2 in~\cite{BL})
$$
\lim_{n\to\infty} \frac{1}{n}  \log\left(  
\M_n
\right)= \lim_{n\to\infty} \frac{1}{n}\log \| T_nT_{n-1}...T_2T_1\|=\mu_1\in(-\infty,0)
$$
almost surely. Therefore, 
$$
\max_{j=1,2,...,d} \| PA^{(n)}_j\| \sim C e^{\mu_1 n}
$$ 
almost surely as~$n\to \infty$.
\end{proof}

\section{Random triangles revisited}
The goal of this Section is to show that in three-dimensional case the projection of the ``middle'' vertex on the largest side of the triangle converges in distribution, thus generalizing the result of Theorem~5 in~\cite{VOL}; our main results is presented in Theorems~\ref{lemconvtoeta} which follows later in the Section. We also evaluate the speed of convergence to flatness in Theorem~\ref{mainthm4}, as well as study some examples; in particular, we strengthen the result of Theorem~4 in~\cite{VOL}.

Since~$x\in{\cal L}$ defined by~\eqref{eqlspace} we can restrict our attention just to the first~$d-1$ coordinates of~$x$. Let us introduce the norm
$$
\|x\|_{\infty}=\max_{j=1,\dots,d} \| x_j\|=\max\{|x_1|,|x_2|,...,|x_{d-1}|,|x_1+...+x_{d-1}|\}.
$$ 
and for each~$x$ in the unit ball~$B_{\infty}=\{(x_1,...,x_{d-1})\in\R^{d-1}: \|x\|_{\infty}=1\}$ 
the map~$\widehat{T}:\ B_{\infty} \to B_{\infty}$ 
by
$$
\widehat{T}(x)=\frac{1}{\|Tx\|_{\infty}}\, Tx.
$$ 
Notice that~$\left\{\widehat{T^{(n)}}(x)\right\}_{n\ge 1}$ is a Markov chain which can be considered as a system of iterated random functions in the sense mentioned 
in~\cite{DF, KAI}. We will use the following result implied from Lemma~2.5 in~\cite{KAI}:

\begin{lemma}
Let~$D_\epsilon=\{(x,y) : x,y \in B_{\infty}, \|x-y\|_{\infty}\le \epsilon   \}$. If \begin{equation}\label{kaijser}
\displaystyle\limsup_{n\to\infty} \sup_{(x,y)\in D_\epsilon}\E\left\|\widehat{T^{(n)}}(x) -\widehat{T^{(n)}}(y) \right\|_{\infty}\to 0
\end{equation} 
as~$\epsilon\to 0$ then~$\widehat{T^{(n)}}(x)$ weakly converges to some random vector.
\end{lemma}

Here is a very important result.
\begin{lemma}\label{lemweak}
Assume that Assumption~\ref{Asu1} and~\ref{Asu2} are fulfilled then~$\left(\max_{j=1,\dots,d}\left\|x^{(n)}_j\right\|\right)^{-1}x^{(n)}$ converges in distribution to some random vector as~$n\to\infty$.
\end{lemma}

\begin{proof}
Assume that all the points~$x$, $y$, etc., belong to~$B_{\infty}$ unless stated otherwise.
Next, w.l.o.g.\ assume that~$\|T^{(n)}x\|_{\infty} \le \|T^{(n)}y\|_{\infty}$, then we have
\begin{align*}
& \|  \widehat{T^{(n)}}(x) -\widehat{T^{(n)}}(y) \|_{\infty}=\frac{1}{\|T^{(n)}x\|_{\infty}}\left\| T^{(n)}\left(x-\frac{\|T^{(n)}x\|_{\infty}}{\|T^{(n)}y\|_{\infty}} y \right)\right\|  \le \frac{\|T^{(n)}\|_{\infty}}{\|T^{(n)}x\|_{\infty}}\left\|x-\frac{\|T^{(n)}x\|_{\infty}}{\|T^{(n)}y\|_{\infty}} y \right\|_{\infty} \\
& \le \frac{\|T^{(n)}\|_{\infty}}{\|T^{(n)}x\|_{\infty}} \left(  \|x-y \|_{\infty}+\left( 1-\frac{\|T^{(n)}x\|_{\infty}}{\|T^{(n)}y\|_{\infty}}  \right)\|y\|_{\infty}\right)\\
 &\le \frac{\|T^{(n)}\|_{\infty}}{\|T^{(n)}x\|_{\infty}} \|x-y \|_{\infty}+\frac{\|T^{(n)}\|_{\infty}}{\|T^{(n)}x\|_{\infty}\cdot \|T^{(n)}y\|_{\infty}}\left( \|T^{(n)}y\|_{\infty}-\|T^{(n)}x\|_{\infty}\right) \\
& \le  \frac{\|T^{(n)}\|_{\infty}}{\|T^{(n)}x\|_{\infty}} \|x-y \|_{\infty} + \frac{\|T^{(n)}\|_{\infty}}{\|T^{(n)}x\|_{\infty}\cdot \|T^{(n)}y\|_{\infty}}
\left( \|T^{(n)}(y-x) \|_{\infty}
\right)  \  \text{ (by the triangle inequality)}\\
& \le  2 \frac{\|T^{(n)}\|^2_{\infty}}{\|T^{(n)}x\|_{\infty}\cdot \|T^{(n)}y\|_{\infty}}\|x-y \|_{\infty} ,
\end{align*}
where~$\|T\|_{\infty}=\sup_{x\in B_{\infty}}\|Tx\|_{\infty}, \|T\|=\sup_{\|x\|=1}\|Tx\|$ are the usual operator norms. Therefore, since all the norms on finite dimensional spaces are equivalent, there exists a non random constant~$r>0$ such that
\begin{equation}\label{Thatbound0}
 \|  \widehat{T^{(n)}}(x) -\widehat{T^{(n)}}(y) \|_{\infty} \le  r \cdot \frac{\|T^{(n)}\|^2}{\|T^{(n)}x\|\|T^{(n)}y\|}\|x-y \|_{\infty} 
\end{equation}
On the other hand, by Theorem~III.3.1 in~\cite{BL}, for almost all~$\omega$, the exist one-dimensional linear space~$V(\omega)\subset \mathbb{R}^{d-1}$ which is the range of limit points of~$\| T_1(\omega)\dots T_n(\omega)\|^{-1}T_1(\omega)\dots T_n(\omega)$. By the proof of Proposition~III.3.2 in~\cite{BL} if a sequence~$\{x_n\}_{n\ge1}\subset B_{\infty}$ converges to~$x$ and~$\zeta_{x}(\omega)$ is the orthogonal projection of~$x$ onto~$V(\omega)$ then 
\begin{align}\label{ineqzeta}
\limsup_{n\to\infty}\frac{\|T^{(n)}\|}{\|T^{(n)}x_n\|} \le \|\zeta_{x}\|^{-1}\ \ \text{a.s.}
\end{align}
and 
\begin{align}\label{probzeta}
\P\left(\|\zeta_x\|=0\right)=0.
\end{align}
Therefore, we obtain that
\begin{equation}\label{Thatbound}
\limsup_{n\to\infty}\|  \widehat{T^{(n)}}(x) -\widehat{T^{(n)}}(y) \|_{\infty} \le r  \|\zeta_{x}\|^{-1}  \|\zeta_{y}\|^{-1} \|x-y \|_{\infty} \ \ {\rm a.s.}
\end{equation}

Let us now verify the condition~\eqref{kaijser}. We have
\begin{align*}
\E\left\|\widehat{T^{(n)}}(x) -\widehat{T^{(n)}}(y) \right\|_{\infty}&\le \E\left(\left\|\widehat{T^{(n)}}(x) -\widehat{T^{(n)}}(y) \right\|_{\infty} \mathbf{1}_{\{ \frac{\|T^{(n)}\|^2}{\|T^{(n)}x\|\|T^{(n)}y\|}\ge \frac{1}{4r\epsilon} \}}\right) + \\
&  + \E\left(\left\|\widehat{T^{(n)}}(x) -\widehat{T^{(n)}}(y) \right\|_{\infty} \mathbf{1}_{\{\frac{\|T^{(n)}\|^2}{\|T^{(n)}x\|\|T^{(n)}y\|}\le \frac{1}{4r\epsilon} \}}\right)
=:{\sf (I)}+{\sf (II)}
\end{align*}

To bound~${\sf (I)}$, observe that~$\left\|\widehat{T^{(n)}}(x) -\widehat{T^{(n)}}(y) \right\|_{\infty}\le 2$
and therefore~${\sf (I)}\le 2\, 
\P \left(\frac{\|T^{(n)}\|^2}{\|T^{(n)}x\|\|T^{(n)}y\|}\ge \frac{1}{4r\epsilon} \right)$.
Suppose
$$
\limsup_{n\to\infty}\sup_{(x,y)\in D_\epsilon}
\P \left(\frac{\|T^{(n)}\|^2}{\|T^{(n)}x\|\|T^{(n)}y\|}\ge \frac{1}{4r\epsilon} \right)\not\to 0
\ \text{\ as $\eps\to 0$.}
$$
Then there exists a~$\beta>0$ and a decreasing sequence~$\eps_k\downarrow 0$ such that this~$\limsup_{n\to\infty} 
\sup_{(x,y)\in D_{\epsilon_k}}\P(\dots\ge (4r\epsilon_k)^{-1})
 >\beta$ for each~$k$; therefore for each~$k$ there is a sequence~$n_i^{(k)}$, $i=1,2,\dots$ such that
\begin{align}\label{eqPbeta}
\delta(k):=\P \left(\frac{\|T^{(n_i^{(k)})}\|^2}{\|T^{(n_i^{(k)})}x_{n_i^{(k)}}\|\|T^{(n_i^{(k)})}y_{n_i^{(k)}}\|}\ge \frac{1}{4r\eps_k} \right)>\beta
\ \text{ for all }i=1,2,\dots.
\end{align}
Let~$m_k=n_k^{(k)}$. Without loss of generality assume that~$x_{m_k}\to x_*\in B_\infty$ and~$y_{m_k}\to y_*\in B_\infty$; since~$B_\infty$ is compact we can always choose a convergence subsequence.

By~\eqref{ineqzeta} we have
$$
\delta_x(k):=\P\left(
\frac{\|T^{(m_k)}\|^2}{\|{T^{(m_k)}}x_{m_k}\|}\ge
2 \|\zeta_{x_*}\|^{-1}\right)\to 0, \ \
\delta_y(k):=\P\left(
\frac{\|T^{(m_k)}\|^2}{\|{T^{(m_k)}}y_{m_k}\|}\ge
2 \|\zeta_{y_*}\|^{-1}\right)\to 0, \ \
$$
as~$k\to\infty$.
Hence
\begin{align*}
\delta(k)&\le \delta_x(k)+\delta_y(k)
+\P \left(4 \|\zeta_{x_*}\|^{-1} \|\zeta_{y_*}\|^{-1}
\ge \frac{1}{4r\eps_k} \right)
=
\delta_x(k)+\delta_y(k)
+\P \left( \|\zeta_{x_*}\|\cdot  \|\zeta_{y_*}\|
\le 16r\eps_k \right)
\\
&\le 
\delta_x(k)+\delta_y(k)
+\P \left( \|\zeta_{x_*}\|\le 4\sqrt{r\eps_k} \right)
+\P \left( \|\zeta_{x_*}\|\le 4\sqrt{r\eps_k} \right)
\to 0
\end{align*}
by~\eqref{probzeta}, leading to a contradiction with~\eqref{eqPbeta}.

On the other hand, if~$(x,y)\in D_\epsilon$ and~$\frac{\|T^{(n)}\|^2}{\|T^{(n)}x\|\|T^{(n)}y\|}\le \frac{1}{4r\epsilon}$ then the inequality~\eqref{Thatbound} implies that~$\|\widehat{T^{(n)}}(x) -\widehat{T^{(n)}}(y)\|_{\infty}\le \frac{1}{4}$, hence
\begin{align*}
& \sup_{(x,y)\in D_\epsilon} {\sf(II)} = \sup_{(x,y)\in D_\epsilon} \E\left(\left\|\widehat{T^{(n)}}(x) -\widehat{T^{(n)}}(y) \right\|_{\infty} \mathbf{1}_{\{ \frac{\|T^{(n)}\|^2}{\|T^{(n)}x\|\|T^{(n)}y\|}\le \frac{1}{4r\epsilon} \}}\right)\\
& \le \sup_{(x,y)\in D_\epsilon}\E\left(\left\|\widehat{T^{(n)}}(x) -\widehat{T^{(n)}}(y) \right\|_{\infty} \mathbf{1}_{\{\|\widehat{T^{(n)}}(x) -\widehat{T^{(n)}}(y)\|_{\infty}\le \frac{1}{4} \}}\right) \\
& \le {\sf Const}\cdot \sup_{x,y\in B_{\infty}}\E \delta\left(T^{(n)}\overline{x},T^{(n)}\overline{y}\right)  \end{align*}
where the last inequality holds since~$\| u-v \|_{\infty}\le {\sf Const}\cdot  \delta(\overline{u},\overline{v})$ for any vectors~$u,v$ such that the angle between~$u$ and~$v$ is smaller than~$\frac{\pi}{2}$. Finally, from the proof of Theorem~III.4.3 in~\cite{BL}, we have 
$$
\limsup_{n\to\infty} \sup_{x,y\in B_{\infty}}\E \delta\left(T^{(n)}\overline{x},T^{(n)}\overline{y}\right) =0.
$$
Therefore the condition~\eqref{kaijser} is fulfilled. 
\end{proof}

From now on assume that~$d=3$. Following~\cite{VOL}, for each~$n\ge 0$ rescale the triangle~$A_1^{(n)}A_1^{(n)}A_3^{(n)}$ to a new triangle~$B_1^{(n)}B_1^{(n)}B_3^{(n)}$ such that its longest edge has length~$1$, its vertices are reordered in a way that~$B_1^{(n)}B_2^{(n)}\ge  B_3^{(n)} B_1^{(n)}\ge B_2^{(n)}B_3^{(n)}$, and let the Cartesian coordinates of vertices be~$B_1^{(n)}=(0,0), B_2^{(n)}=(1,0), B_3^{(n)}=\theta_n=(g_n,h_n)$; formally 
\begin{align}\label{eqArea}
h_n=\frac{2 \Area(L_n)}{\M_n^2}
\end{align}
is the height of the rescaled triangle, corresponding to the largest side with length~$\M_n$, formally defined by~\eqref{eqMndefined}.
Without loss of generality, we can also  assume that~$A_1^{(0)} \equiv B_1^{(0)},A_2^{(0)} \equiv B_2^{(0)}, A_3^{(0)} \equiv B_3^{(0)}$.

\begin{theorem}\label{lemconvtoeta}
Assume that Assumption~\ref{Asu1} is fulfilled then~$g_n$  converges in distribution to some random variable~$\eta\in[1/2,1]$.
\end{theorem}
\begin{proof}
Recall that~$(x_i^{(n)},y_i^{(n)})$, $i=1,2,3$ denote the coordinates of the vectors corresponding to the three sides of the triangle.
Since they are asymptotically collinear, $g_n$ has the same limit as 
$$
g_{x,n}=f(x_1^{(n)},x_2^{(n)},x_3^{(n)})
$$
where $f(a,b,c)$ is the ratio between the second largest amongst~$\{|a|,|b|,|c|\}$ and the largest of them; in fact, $|g_{x,n}-g_n|\le h_n\to 0$ a.s.\ from an elementary geometric observation.
The only  problem which could arise is if the triangle is (nearly) vertical; however this does not happen for large $n$ a.s.\ by Lemma~II.4.2 from~\cite{BL} which says (equivalently) the the direction of the limiting flat triangle has a continuous distribution.

Since in our case~$a+b+c=0$,  we can write $f$ as
$$
f(a,b,c)=1-\frac{\min\{|a|,|b|,|c|\}}{\max\{|a|,|b|,|c|\}}
$$
Since the function~$f(\cdot)$ is continuous, and the vector~$x^{(n)}/\|x^{(n)}\|$ converges weakly by Lemma~\ref{lemweak}, the result follows. 
\end{proof}

\begin{theorem}\label{mainthm4}
Suppose that  Assumption~\ref{Asu3} is fulfilled. Then 
$$
\lim_{n\to\infty}\frac{1}{n}\log (h_n)=\E(\log(\det(T_1)))-2\int_{1/2}^1\zeta(x,0)d\P_{\eta}(x),\quad \text{a.s.}
$$
where~$\theta_n=(g_n,h_n)$ is defined just above~\eqref{eqArea}, $\eta$ is the weak limit of~$g_n$,  $\P_\eta$ is its probability measure,
and 
$$
\zeta(x,y)=\E\left(\log(
\M_1
)\ |\ \theta_0 = (x,y) \right).
$$
\end{theorem}
\begin{proof}
We have the following relation
$$
h_n= h_{n-1}\cdot \frac{
\M_{n-1}
^2}
{
\M_n^2}\cdot \det(T_n)
$$
which implies that
$$
\frac{1}{n}\log (h_n) =\frac{1}{n}
\sum_{j=1}^n\log( \det(T_j))-\frac{2}{n}\log (\M_n)
+O\left(\frac 1n\right).
$$
Suppose that Assumption~\ref{Asu3} is fulfilled. By the strong law of large numbers  and equation~\eqref{eqsumofsv} we have
$$
\lim_{n\to\infty}\frac{1}{n}\sum_{j=1}^n\log( \det(T_j))=\E(\log(\det(T_1)))=\mu_1+\mu_2 \quad \text{a.s.}
$$
By Lemma~\ref{lemtildeS} 
$$
\lim_{n\to\infty}\frac{1}{n}\log (\M_n)\to \mu_1 \ \ \text{a.s.}
$$
so that
$$
\lim_{n\to\infty}\frac{1}{n}\log (h_n) =\mu_2-\mu_1\ \quad \text{a.s.}
$$
On the other hand, we have 
$$
\frac{1}{n}\log (h_n) =\frac{1}{n}\sum_{j=1}^n\log( \det(T_j)) -\frac{2}{n}\sum_{j=1}^{n}\log \left(\frac{\M_j}{\M_{j-1}}\right)+ \frac{1}{n}\log (h_0).
$$
Let~$\P_n(d\theta|\ \theta_0)$ be the conditional probability measure of~$\theta_n$ on~$\theta_0$. We have
$$
\E\left(\frac{1}{n}\sum_{j=1}^n\log \left(\frac{\M_j}{\M_{j-1}}\right) \ | \ \theta_0\right)=\sum_{j=1}^n \frac{1}{n} \E\left(\log \left(\frac{\M_j}{\M_{j-1}}\right) \ | \ \theta_{0}\right)=\int \zeta (\theta) \tilde{\P}_n(d\theta|\ \theta_0),$$
where we denote
$\zeta(\theta)=\E\left(\log(\M_1)\ |\ \theta_0=\theta \right)$
and
$\displaystyle \tilde{\P}_n(d\theta|\ \theta_0)= \frac{1}{n} \sum_{j=1}^{n-1}{\P}_j(d\theta|\ \theta_0).$

We already know that~$h_n\to 0$ almost surely and~$g_n$ converges in distribution to some random variable~$\eta$ taking value on~$(1/2,1)$, therefore~$\theta_n=(g_n,h_n)$ converges in distribution to~$(\eta,0)$ as~$n\to\infty$. Since~$\zeta(x,0)$ is a continuous function of~$x$ on~$(1/2,1)$, using Ces\`aro mean result we have
$$
\lim_{n\to\infty}\int \zeta (\theta) \tilde{\P}_n(d\theta|\ \theta_0)=\lim_{n\to\infty}\int \zeta (\theta) {\P}_n(d\theta|\ \theta_0)=\int_{1/2}^1\zeta(x,0)d \P_{\eta}(x).
$$
where~$\P_\eta$ is the probability measure of~$\eta$. Therefore
\begin{align}\label{eq_speed}
\lim_{n\to\infty}\frac{1}{n}\log (h_n)=\E(\log(\det(T_1)))-2\int_{1/2}^1\zeta(x,0)
d\P_{\eta}(x).
\end{align}
\end{proof}

\textbf{Example 1.} Let us consider the case when random variables~$\xi_{1},\xi_{2},\xi_{3}$ are uniformly distributed on~$(0,1)$, notice that~$\theta_n=(g_n,h_n)$ converges in distribution to~$(U,0)$, where~$U$ is uniformly distributed on~$(\frac{1}{2},1)$, see~\cite{VOL}. We easily obtain that
\begin{align*}
\E\left(\log(\det(T_1))\right) =\E\left(\log\left((1-\xi_1)(1-\xi_2)(1-\xi_3)+\xi_1\xi_2\xi_3\right)\right)&=\frac{-24+\pi^2}{9}.
\end{align*}
and
$$
\zeta(x,0)=\frac{x \left(2 x^2 \log (x)-5 x+5\right)-2 (x-1)^3 \log (1-x)}{6 (x-1) x},
$$
hence 
\begin{align*}
\int_{1/2}^1\zeta(x,0)dx&=\frac{-15+\pi^2}{18}
\end{align*}
and we can conclude from~\eqref{eq_speed} that 
$$
h_n\sim C e^{-\frac{\pi^2-6}{9}n}
$$ 
as~$n\to\infty$ in the sense that~$\frac 1n \log h_n\to -\frac{\pi^2-6}{9}\approx -0.43$, thus strengthening the result of Theorem~4 in~\cite{VOL}.

\vskip 5mm

\textbf{Example 2.} 
Suppose that~$\xi_1,\xi_2,\xi_3$ have a continuous distribution with density symmetric around~$\frac 12$, i.e.~$p(1-x)=p(x)$.
Let~$x\in(0,1)$ and  set~$x_1=x \xi_1, x_3=x+(1-x)\xi_3, x_2=\xi_2$ and~$y_1\le y_2\le y_3$ be the triple~$x_1,x_2,x_3$ sorted in the increasing order. For~$z<x$, we have
$$
\P\left( \frac{ y_2-y_1}{y_3-y_1}<z \right)=  I_1(z,x)+I_2(z,x)
$$
where
\begin{align*}
 & I_1(z,x) = \P\left( \frac{ y_2-y_1}{{x_3}-y_1}<z; y_1<y_2<x<x_3 \right) = \P\left(  y_2<z x_3+(1-z)y_1; y_1<y_2<x<x_3  \right)\\ 
 & = \P\left(  y_1<y_2<x; z x_3+(1-z)y_1 >x  \right) +\P\left( y_1<y_2 <  z x_3+(1-z)y_1; z x_3+(1-z)y_1<x \right) \\
 & = \P\left(  y_1<y_2<x;\ \frac{x-zx_3}{1-z}< y_1<x;\ x<x_3<1 \right) \\
 & + \P\left( y_1<y_2 <  z x_3+(1-z)y_1; z x_3+(1-z)y_1<x;\ 0<y_1<\frac{x-zx_3}{1-z};\   x<x_3<1 \right)  \\
 &=  \int_x^1dx_3\int_{\frac{x-z x_3}{1-z}}^x dy_1 \int_{y_1}^x \left[ \frac{1}{x} p\left(\frac{y_1}{x}\right)p(y_2) + \frac{1}{x} p\left(\frac{y_2}{x}\right)p(y_1) \right]\frac{1}{1-x}p\left(\frac{x_3-x}{1-x}\right) dy_2\\
& + \int_x^1dx_3 \int_0^{\frac{x-z x_3}{1-z}} dy_1 \int_{y_1}^{z x_3+(1-z)y_1}\left[ \frac{1}{x} p\left(\frac{y_1}{x}\right)p(y_2) + \frac{1}{x} p\left(\frac{y_2}{x}\right)p(y_1) \right]\frac{1}{1-x}p\left(\frac{x_3-x}{1-x}\right) dy_2,\\
\end{align*}
and
\begin{align*}
& I_2(z,x)=  \P\left( \frac{ y_2-x_1}{{y_3}-x_1}<z; x_1<x<y_2<y_3 \right) = \P\left( y_3> \frac{y_2-(1-z)x_1}{z}; x_1<x<y_2<y_3 \right)\\
& = \P\left(\frac{y_2-(1-z)x_1}{z} <y_3<1; \frac{y_2-(1-z)x_1}{z}<1;  y_2>x \right)\\
& =  \P\left(\frac{y_2-(1-z)x_1}{z} <y_3<1; x<y_2< (1-z)x_1+z ;(1-z)x_1+z>x \right)\\
& = \P\left(\frac{y_2-(1-z)x_1}{z} <y_3<1; x<y_2< (1-z)x_1+z ; \frac{x-z}{1-z}<x_1<x\right)\\
&=  \int_{\frac{x-z}{1-z}}^x dx_1\int_x^{(1-z)x_1+z} dy_2 \int_{\frac{y_2-(1-z)x_1}{z}}^{1} \\ & \times\left[ \frac{1}{1-x}p\left(\frac{y_3-x}{1-x}\right)p(y_2)+\frac{1}{1-x}p\left(\frac{y_2-x}{1-x}\right)p(y_3)\right] \frac{1}{x}p\left(\frac{x_1}{x}\right)dy_3 .   
\end{align*}
For~$z>x$, by the symmetric property, we have
$$
\P\left( \frac{ y_2-y_1}{y_3-y_1}<z \right)=  I_1(1-z,1-x)+I_2(1-z,1-x).
$$

Let~$\eta$ be the invariant distribution defined in Theorem~\ref{lemconvtoeta}. Assume that~$2\eta-1$ has the density~$\phi(x)$, then~$\phi(x)$ is the unique solution of the following integral equation:
\begin{align}\label{eqinvarianteta}
\int_0^z\phi(x)dx&=\int_0^z\left[ I_1(1-z,1-x)+I_2(1-z,1-x) \right] \phi(x)dx\nonumber \\
&+ \int_z^{1}\left[ I_1(z,x) +I_2(z,x)\right]\phi(x)dx
\end{align}
since one can look, for example, at the linear projections of the vertices of the triangle, see also~\cite{VOL}.

Now fix a positive integer~$n$, and additionally assume that~$\xi_1,\xi_2,\xi_3$ are independent Beta$(n,n)$ distributed random variables, i.e.\ their  density function is given by
$$
p_n(\xi)= \begin{cases}
\frac{\xi^{n-1}(1-\xi)^{n-1}}{ {\rm B}(n,n)},& \xi\in(0,1)\\
0, & \text{otherwise}
\end{cases}
\quad {\rm where} \  
    {\rm B}(x,y) = \int_0^1t^{x-1}(1-t)^{y-1}\,\mathrm{d}t
$$
is the usual  Beta function. Let the corresponding invariant distribution~$\phi_n(x)$ be defined by~\eqref{eqinvarianteta}.

Using a computer algebra system, e.g.\ Mathematica\texttrademark\ or Maple\texttrademark, one can check that the solution to~\eqref{eqinvarianteta} for~$n=1,2,3,4,5$ are given by
\begin{align*}
\phi_1(z)&=1,\\
\phi_2(z)&=\frac{6}{7} \left((1-z) z+1\right) ,\\
\phi_3(z)&=\frac{30}{143} \left(3 (1-z)^2 z^2+4 (1-z) z+4\right) ,\\
\phi_4(z)&= \frac{140}{4199} \left(13 (1-z)^3 z^3+22 (1-z)^2 z^2+25 (1-z) z+25\right),\\
\phi_5(z)&= \frac{6174}{7429} \left(\frac{17}{49} (1-z)^4 z^4+\frac{5}{7} (1-z)^3 z^3+\frac{13}{14} (1-z)^2 z^2+(1-z) z+1\right).
\end{align*}
We conjecture that in the general case~$\phi_n(z)$ is also a mixture of some Beta distributions, that is, there exist non-negative constants~$c_1,c_2,...,c_{n}$ summing up to~$1$ such that 
$$
\phi_n(z)=\sum_{j=1}^{n}c_{j}\frac{z^{j-1}(1-z)^{j-1}}{B(j,j)}
$$
but unfortunately we cannot prove this fact.
\begin{figure}[!ht]
\centering\includegraphics[scale=0.6]{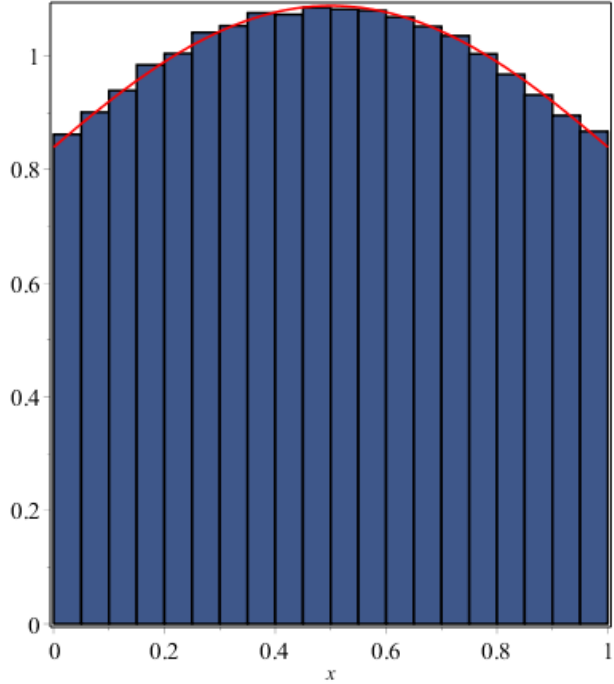} 
\caption{For~$\xi\sim {\sf Beta}(3,3)$, one can see the similarity between the histogram of~$\{2g_j-1, j=1,2,..., 10^6\}$ obtained from simulation and the plot of~$\{\phi_3(x), x\in[0,1]\}$.}
\end{figure}

\section{Generalizations and open problems}\label{sec_gener}
Let~$\xi_1,\xi_2,\dots,\xi_d$  be the random variables governing how the sides of the~$d$-polygon are split at each iteration. Throughout the paper we have assumed that~$\xi_j$, $j=1,\dots,d$ are i.i.d. However, if one looks at the proofs, one can see that the independence assumption can be substantially relaxed without any change in the proofs.
Indeed, let~$\bar\xi=(\xi_1,\xi_2,\dots,\xi_d)$ be the random variable describing the splitting proportions of the sides of the polygon. Assume that
\begin{itemize}
\item[(i)]
$\P(0<\xi_i<1)=1$ for all~$i$;
\item[(ii)]
there are two distinct numbers~$a,b\in(0,1)$ such that all~$2^d$ points of the form~$x=(x_1,\dots,x_d)\in\R^d$, where each~$x_i=a$ or~$=b$, belong to the support of~$\bar\xi$;
\item[(iii)]
if~$d$ is even then~$\xi_1\xi_2\dots \xi_d\ne (1-\xi_1)(1-\xi_2)\dots (1-\xi_d)$ a.s.
\end{itemize}
Then Conjecture~\ref{conj} is fulfilled (observe that we still suppose that random variables~$\bar\xi$ are drawn in i.i.d.\ manner for each iteration).

We also strongly feel that assumption (iii)
is, in fact, superfluous, so the result will hold even if some matrices are degenerate. Indeed, intuitively, when some of the matrices in the product are not full rank, this should even be helpful for the convergence to lower-dimensional subspaces. However, in this case we would clearly not be able to form a group containing all the matrices in the support of the measure and hence cannot use the standard results from the random matrix theory.

Another possible generalization of our model is to higher dimensional spaces, e.g.\ random subdivision of tetrahedrons in~$\R^3$. We are currently working on showing similar results in this case.

\subsection*{Acknowledgement}
We would like to thank the anonymous referees for useful comments and suggestions. S.V.~research is partially supported by Swedish Research Council grant~VR~2014--5157 and Crafoord foundation; M.N.~research is partially supported by 
Thorild Dahlgrens and Folke Lann\'ers funds and Crafoord foundation.

\end{document}